\documentclass[11pt]{article}
\usepackage{mathrsfs}
\usepackage{amsfonts}
\usepackage{amssymb}
\usepackage{amssymb,amsmath,amsthm,amsfonts}

\def\sqr#1#2{{\vcenter{\vbox{\hrule height.#2pt
              \hbox{\vrule width.#2pt height#1pt \kern#1pt \vrule width.#2pt}
          \hrule height.#2pt}}}}
\def\dbF{\hbox{\rm l\negthinspace F}}

\def\dbP{\hbox{\rm l\negthinspace P}}

\def\dbR{\hbox{\rm l\negthinspace R}}

\def\a{\alpha}
\def\b{\beta}

\def\e{\varepsilon}

\def\l{\lambda}

\def\cF{{\cal F}}

\def\cU{{\cal U}}

\def\sqr#1#2{{\vcenter{\vbox{\hrule height.#2pt
              \hbox{\vrule width.#2pt height#1pt \kern#1pt \vrule width.#2pt}
              \hrule height.#2pt}}}}

\def\dbR{{\mathop{\rm l\negthinspace R}}}

\def\3n{\negthinspace \negthinspace \negthinspace }
\def\2n{\negthinspace \negthinspace }
\def\1n{\negthinspace }

\def\dbF{{\mathop{\rm l\negthinspace F}}}

\def\dbP{{\mathop{\rm l\negthinspace P}}}
\def\dbR{{\mathop{\rm l\negthinspace R}}}
\def\={\buildrel \triangle \over =}

\def\a{\alpha}
\def\b{\beta}

\def\e{\varepsilon}

\def\l{\lambda}

\def\cF{{\cal F}}

\def\cU{{\cal U}}

\def\no{\noindent}

\def\ms{\medskip}

\def\q{\quad}
\def\qq{\qquad}

\def\esssup{\mathop{\rm esssup}}
\def\max{\mathop{\rm max}}
\def\min{\mathop{\rm min}}

\def\sup{\mathop{\rm sup}}

\def\esssup{\hbox{\rm ess$\,$\rm sup$\,$}}

\def\ae{\hbox{\rm a.e.{ }}}
\def\as{\hbox{\rm a.s.{ }}}

\def\|{\Big |}
\def\({\Big (}
\def\){\Big )}
\def\[{\Big[}
\def\]{\Big]}
\def\be{\begin{equation}}
\def\bel{\begin{equation}\label}
\def\ee{\end{equation}}
\def\bt{\begin{theorem}}
\def\bcd{\begin{condition}}
\def\ecd{\end{condition}}
\def\et{\end{theorem}}
\def\bc{\begin{corollary}}
\def\ec{\end{corollary}}
\def\bde{\begin{definition}}
\def\ede{\end{definition}}
\def\bl{\begin{lemma}}
\def\el{\end{lemma}}
\def\bp{\begin{proposition}}
\def\ep{\end{proposition}}
\def\br{\begin{remark}}
\def\er{\end{remark}}
\def\ba{\begin{array}}
\def\ea{\end{array}}
\def\ed{\end{document}}

\def\square#1{\vbox{\hrule\hbox{\vrule height#1%
     \kern#1\vrule}\hrule}}
\def\rectangle#1#2{\vbox{\hrule\hbox{\vrule height#1%
     \kern#2\vrule}\hrule}}
\def\qed{\hfill \vrule height7pt width3pt depth0pt}

\newcommand{\ind}{1\!\!1}

\font\tenbb=msbm10 \font\sevenbb=msbm7 \font\fivebb=msbm5

\newfam\bbfam
\scriptscriptfont\bbfam=\fivebb \textfont\bbfam=\tenbb
\scriptfont\bbfam=\sevenbb

\newtheorem{lemma}{Lemma}[section]
\newtheorem{remark}{Remark}[section]

\newtheorem{theorem}{Theorem}[section]
\newtheorem{corollary}{Corollary}[section]

\newtheorem{definition}{Definition}[section]
\newtheorem{proposition}{Proposition}[section]
\newtheorem{condition}{Assumption}[section]
\def \l{\left}
\def \r{\right}
\newcommand{\Prob}{\operatorname{\mathbb{P}}}

\makeatletter
   
   \@addtoreset{equation}{section}
\makeatother

\begin{document}

\title{\bf A characterization of sub-game perfect Nash equilibria  for SDEs of mean field type}

\author{Boualem Djehiche\thanks{Department of Mathematics, KTH Royal Institute of Technology, SE-100 44 Stockholm, Sweden.  E-mail: boualem@math.kth.se.}\,\,\footnote{Financial support from the Swedish Export Credit Corporation (SEK)   is gratefully acknowledged. Many thanks to Tomas Bj\"ork and Georges Zaccour for fruitful discussions and their insightful comments.}\qquad and \qquad Minyi Huang\thanks{ School of Mathematics and Statistics, Carleton University,
1125 Colonel By Drive, Ottawa, ON K1S 5B6, Canada. E-mail:  mhuang@math.carleton.ca.}}

\date{}

\maketitle

\begin{abstract} 
We study a class of dynamic decision problems of mean field type with time inconsistent cost functionals, and derive a stochastic maximum principle to characterize subgame perfect Nash equilibrium points. Subsequently, this approach is extended to a mean field game to construct decentralized strategies and obtain an estimate of their performance.
\end{abstract}

\begin{center}
\begin{minipage}{14.6cm}
\end{minipage}
\end{center}

\begin{center}
\begin{minipage}{14.6cm}
{\bf Abbreviated title.} Nash equilibria  for SDEs of mean field type
\end{minipage}
\end{center}

\begin{center}
\begin{minipage}{14.6cm}
\end{minipage}
\end{center}

\section{Introduction}

In dynamic decision making problems a policy is time consistent if whenever it is optimal at time $t$, it remains optimal when implemented at a later time $s>t$. In optimal control this is known as the Bellman principle. A time inconsistent policy need not be optimal at later time $s>t$, even if it is optimal at time $t$. Time inconsistency  occurs for example when a hyperbolic discount rate is preferred to an exponential discount rate or when the performance criterion is a nonlinear function of the expected utility such as the variance in the standard Markowitz investment problem. For a recent review of time consistency in dynamic decision making problems we refer to Ekeland and  Lazrak (2006), Ekeland and Pirvu (2008), and Zaccour (2008).

 In his work on a deterministic Ramsay problem, Strotz (1955) was the first to formulate the dynamic time inconsistent decision problem as a game theoretic problem where it is natural to look for subgame perfect Nash equilibria. Pollak (1968), Phelps and Pollak (1968), Peleg and Menahem (1973) and Goldman (1980) extended this framework to discrete and continuous time dynamics. The recent works by Ekeland and Lazrak (2006) and Ekeland and Pirvu (2008) apply this game theoretic approach to an optimal investment and consumption problem under hyperbolic discounting for deterministic and stochastic models. Among their achievements, they provide a precise definition of the equilibrium concept in continuous time, using a Pontryagin type ``spike variation" formulation (that we recall in Section \ref{sec:notation} below) and derive among other things, an extension of the Hamilton-Jacobi-Bellman (HJB) equation along with a verification theorem that characterizes Markov (or feedback type) subgame perfect Nash equilibria. Their work is extended by Bj\"ork and Murgoci (2008) and Bj\"ork, Murgoci and Zhou (2014) to performance functions that are nonlinear functions of expected utilities for dynamics driven by a quite general class of Markov processes. Hu {\it et al.} (2012) followed by Bensoussan {\it et al.} (2013) characterize subgame perfect Nash equilibria using a Pontryagin type stochastic maximum principle (SMP) approach to a time inconsistent  stochastic linear quadratic control problem of mean-field type, where the  performance functional is a  conditional expectations with respect to  the history ${\cal F}_t$ of the system up to time $t$. They derive a general sufficient condition for equilibria through a new class of  flows of forward-backward stochastic differential equations (FBSDEs). The properties of this class of  flows of FBSDEs are far from being well understood and deserve further investigation. Both the extended HJB equation provided in Bj\"ork and Murgoci (2008) and  Bj\"ork, Murgoci and Zhou (2014) and  the sufficient condition suggested by Hu {\it et al.} (2012) give explicit expression of the equilibria only in very few cases. In a more recent work, Yong (2013b) studied a class of  linear-quadratic models with very general weight matrices in the cost, and time-consistent equilibrium control is constructed by the stochastic maximum principle approach and Riccati equations. Yong (2013b) also considered closed-loop equilibrium strategies by discretization of time for the game.

In this paper we suggest an SMP approach to time inconsistent decision problems for dynamics that is driven by diffusion processes of mean field type that are not necessarily Markov and whose performance criterion is  a nonlinear function of the conditional expectation of a utility function, given the present location of the state process. We do not condition on the whole history ${\cal F}_t$ of the system as in Hu {\it et al.} (2012) because for all practical purposes, in the best conditions, the decision-maker  can only observe the current state of the system. She can never provide a complete and explicit form of the history ${\cal F}_t$ (which is a $\sigma$-algebra) of the system, simply because this is a huge set of information, except in trivial situations.  Our model generalizes the one studied in Ekeland and  Pirvu (2008) and Bj\"ork {\it et al.} (2008, 2014).

 In the first main result of the paper the subgame perfect Nash equilibria (not necessarily of feedback type) are fully characterized as maximizers of the Hamiltonian associated with the system in a similar fashion as in the SMP for diffusions of mean field type obtained in Andersson and Djehiche (2010) and Buckdahn {\it et al.} (2011). This approach is illustrated by several examples and the explicit solutions are obtained.

 Next, we address the time-inconsistency issue in a mean field game setting which involves $N$ players in decision making.
In a mean field game, the players are individually insignificant and interact via an aggregate effect (called the mean field effect) generated by the population.
There has existed a substantial literature on this class of games.  Huang, Caines, and Malham\'e (2003, 2006, 2007) introduced an approach based on consistent mean field approximations to design decentralized strategies where each player solves a localized optimal control problem by dynamic programming.
 These strategies have an $\varepsilon$-Nash equilibrium property when applied to a large but finite population. Closely related developments were presented by Lasry and Lions (2007) who introduced the name mean field game, and Weintraub, Benkard, and Van Roy (2008) studied oblivious equilibria in a Markov decision setup.
Within the linear quadratic setup, various explicit solutions can be obtained; see e.g.  Huang, Caines, and Malham\'e (2007), Li and Zhang (2008), Bardi (2012), Bensoussan {\it et al.} (2011). Tembine, Zhu and Basar (2011) introduced risk sensitive costs for mean field games, and analyzed the linear exponential quadratic Gaussian model in detail.  For games with dynamics modelled by nonlinear diffusions, Carmora and Delarue (2013) developed a probabilistic approach, and  Kolokoltsov, Li and Yang (2011) presented a very general mean field game modeling framework via nonlinear Markov processes. Gomes, Mohr and Souza (2010) considered games with discrete time and  discrete states.
For additional information, the reader may consult an overview of this area  by Buckdahn, Cardaliaguet, and Quincampoix (2011), and  Bensoussan, Frehse, and Yam (2012).

To display an overall picture of various past developments in a mean field context, we briefly remark on the difference between mean field type optimal control and mean field games. For the former (see e.g. Andersson and Djehiche (2010), Elliot, Li and Ni (2013), Yong (2013a)), there is only a single decision maker who can instantly affect the mean of the underlying state process.
In contrast, a player in a mean field game with all comparably small players (called peers) has little influence on a mean field term such as $X^{(N)}=\frac{1}{N}\sum_{i=1}^N X_i$ although as $N\rightarrow \infty$, $X^{(N)}$ may asymptotically agree with the mean of a representative agent in a uniform population. An exception is games with a major player whose control can affect everyone notably; see e.g. Huang (2010), Nourian and Caines (2013).

So far most existing research on mean filed games deals with time consistent cost functionals. The state feedback strategies based on consistent mean filed approximations are subgame perfect  in the infinite population limit model and so no individual has the incentive to revise its strategy when time moves forward.  In a recent work, Bensoussan, Sung and Yam (2013) considered time-inconsistent quadratic cost functionals in a mean field game with a continuum population and linear dynamics. A so-called time consistent optimal strategy is  derived  based on spike variation which is followed by a consistency condition on the mean field generated by an infinite population.

The mean field game which we will analyze involves nonlinear dynamics and each player is cost coupled with others by their average state $X^{(-i)}=\frac{1}{N-1}\sum_{k\ne i}^N X_k$. Time inconsistency arises from the conditioning in the cost functional. Our approach for strategy design is to use a freezing idea so that the coupling term is approximated by a deterministic function $\bar X$. This naturally introduces an optimal control problem with time inconsistent cost which in turn is handled by the SMP approach.
After finding the equilibrium strategy for the limiting  control problem,   we determine $\bar X$ by a consistency condition. The remaining important issue is to analyze the performance of the obtained strategies when applied by $N$ players.

 The organization of the paper is as follows.  In Section 2, we state the SMP approach for our game problem and the associated adjoint equations.
 Section \ref{sec:smp} characterizes the equilibrium point by an SMP (Theorem \ref{MP1}).
 Section \ref{sec:app} is devoted to some examples illustrating the main results.
In Section \ref{sec:game} we extend the previous results to a system of $N$ decision makers (Theorem \ref{theorem:asy}). Section \ref{sec:proof} provides the proof of Theorem \ref{theorem:asy}.
Section \ref{sec:mflq} presents explicit computations in a  mean field LQG game with time inconsistent costs.

 To streamline the presentation, we only consider the one dimensional case for the state. The extension to the multidimensional case is by now straightforward. For the reader's convenience, we note a convention on notation. The analysis of the mean field game uses $C$ as a generic constant which may change from place to place, but depends on neither the population size $N$ nor the parameter $ \varepsilon$ of the spike variation.

\section{Notation and statement of the problem}\label{sec:notation}

 Let $T>0$ be a fixed time horizon and $(\Omega,{\cal{F}},\dbF, \dbP)$  be a given filtered
probability space  whose filtration $\dbF=\{{\mathcal{F}}_s,\ 0\leq s \leq T\}$
satisfies the usual conditions of right-continuity and completeness, on which a one-dimensional standard
Brownian motion $W=\{W_s\}_{s\geq0}$ is given. We assume that $\dbF$ is the natural filtration of $W$ augmented by $\dbP$-null sets of ${\cal F}.$

 An admissible strategy $u$ is an $\dbF$-adapted and square-integrable process with values in a non-empty subset $U$ of $\dbR$. We denote the set of all admissible strategies over $[0,T]$  by $\mathcal{U}[0,T]$.

  For each admissible strategy $u\in \mathcal{U}[0,T]$, we consider the  dynamics given by the following SDE of mean-field type, defined on $(\Omega,{\cal{F}},\dbF, \dbP)$,
   \begin{equation}\label{SDE}
\left\{\begin{array}{lll}
dX^{u}(s)= b(s,X^{u}(s),E[X^{u}(s)],u(s))ds+\sigma(s,X^{u}(s),E[X^{u}(s)],u(s))dW(s),\, 0<s\le T, \\
X^{u}(0)=x_0 \,\,(\in \dbR),
\end{array}
\right.
\end{equation}

 We consider decision problems related to the following cost functional
\begin{equation}\label{J-1}
J(t,x,u)=E\left[\int_t^T h\l(s,X^{u,t,x}(s),E[X^{u,t,x}(s)],u(s)\r)
ds+g\l(X^{u,t,x}(T),E[X^{u,t,x}(T)]\r)\r],
\end{equation}
associated with the state process $X^{u,t,x}$, parametrized by  $(t,x)\in[0,T]\times \dbR$,  whose dynamics is given by the SDE
 \begin{equation}\label{SDEu}
\left\{\begin{array}{lll}
dX^{u,t,x}(s)= b(s,X^{u,t,x}(s),E[X^{u,t,x}(s)],u(s))ds+\sigma(s,X^{u,t,x}(s),E[X^{u,t,x}(s)],u(s))dW(s),\, t<s\le T, \\
X^{u,t,x}(t)=x \,\,(\in \dbR),
\end{array}
\right.
\end{equation}
where,
\begin{equation*}\begin{array}{lll}
b(s,y,z,v),\,\, \sigma(s,y,z,v),\,\, h(s,y,z,v): \,\,[0,T] \times \dbR\times\dbR\times U\longrightarrow \dbR,\\ g(y,z): \,\,\dbR\times\dbR\longrightarrow \dbR,
\qq s\in[0,T],\, y\in\dbR,\,  z\in\dbR,\, v\in U.
\end{array}
\end{equation*}
We note that $X^{u,0,x_0}=X^{u}$.

 The nonlinearity of the cost functional  (\ref{J-1}) in the term $E[X^{u,t,x}(T)]$ makes the system (\ref{SDEu})-(\ref{J-1}) time-inconsistent in the sense that the Bellman Principle for optimality does not hold, i.e., the $t$-optimal policy
\begin{equation}\label{opt2-u}
u^*(t,x,\cdot):= \arg\min_{u\in\mathcal{U}}J(t,x,u),
\end{equation}
may not be optimal after $t$:

 The restriction of  $u^*(t,x,\cdot)$
on $[t', T]$ is not equal to  $ \arg\min_{u}J(t',x',u)$ for some $t'>t$ when the state process is steered to $x'$ by $u^*$.  Therefore, as noted by Ekeland, Lazrak and Pirvu (2006)-(2008), time inconsistent optimal solutions (although they exist mathematically) are irrelevant in practice. The decision-maker would not implement the $t$-optimal policy at a later time, if he/she is not forced to do so. The review paper by  Zaccour (2008) gives a nice guided tour to the concept of time consistency in differential games.

 Following Ekeland {\it et al.} (2006)-(2008), and Bj\"ork and Murgoci (2008), we may view the problem as a game and look for a {\bf sub-game perfect Nash  equilibrium point} $\hat u$ in the following sense:
\begin{itemize}
\item Assume that all players (selves) $s$, such that $s>t$, use the strategy $\hat u(s)$.
\item Then it is optimal for player (self) $t$ to also use $\hat u(t)$.
\end{itemize}

 When  the players use feedback strategies,  depending on $t$ and on the position $x$ in space, player $t$ will choose a strategy of the form $u(t):=\varphi (t, x)$, where $\varphi$ is deterministic function, so the action chosen by player $t$ is given by the mapping $x\longrightarrow \varphi (t, x)$. The cost to player $t$ is given  by the functional $J(t, x, \varphi)$.  It is clear that $J(t, x, \varphi)$ does not depend on the actions taken by any player $s$ for $s < t$, so in fact $J$ does only depend on the restriction of the strategy $u$ to the time interval $[t, T]$. The strategy $\varphi$ can thus be viewed as a complete description of the chosen strategies of all players in the game.

 If feedback  strategies are to be used, a deterministic function $\hat \varphi:\, [0,T]\times\dbR\longrightarrow U$  is a  sub-game perfect Nash  equilibrium point when the following actions are performed:
\begin{itemize}
\item Assume that all players (selves) $s$, such that $s>t$, use the strategy $\hat \varphi(s,\cdot)$.
\item Then it is optimal for player (self) $t$ to also use $\hat  \varphi(t,\cdot)$.
\end{itemize}


 To characterize the equilibrium strategy $\hat u$, Ekeland {\it et al.} (2007)-(2008) suggest the following definition that uses a ``local" spike variation in a natural way.

  Define the admissible strategy $u^{\e}$ as the ``local" spike variation of a given admissible strategy $\hat u\in\mathcal{U}[0,T]$  over the set $[t,t+\e]$,
\be\label{spike}
u^{\e}(s):=\left\{\begin{array}{ll} u(s),\,\,\,\; s\in [t,t+\e],\\ \\ \hat u(s),\,\,\,\; s\in [t,T]\setminus [t,t+\e],\end{array}\right.
\ee
where $u\in \mathcal{U}[0,T]$ and $t\in[0,T]$ are arbitrarily chosen.

  Hu {\it et al.} (2012) suggest the following open-loop form of the local spike variation:

\be\label{spike-1}
u^{\e}(s):= \hat u(s)+\nu\ind_{[t,t+\e]}(s),\,\,\,\;  s\in [t,T],
\ee
where $\nu\in U$ is arbitrarily chosen such that  for each $s\in[t,t+\e], u^{\e}(s)\in U$ . This is a particular case of  (\ref{spike}), where the arbitrary admissible strategy $u$ over $[t,t+\e]$  is {\it  a deviation} from the equilibrium point $\hat u$  in the direction $\nu$ i.e. $u(s)=\hat u(s)+\nu$. Although this form is  suitable only  when  $U$ is a linear space, it has the advantage of imposing weaker integrability conditions on the admissible strategies since
\be\label{spike-3}
u^{\e}(s)-\hat u(s)=\nu\ind_{[t,t+\e]}(s),\,\,\,\;  s\in [t,T].
\ee

  For either form of local spike variation, we have the following
\medskip
\begin{definition}\label{SGPE} The admissible strategy $\hat u$ is a {\bf sub-game perfect Nash equilibrium}  for the system (\ref{J-1})-(\ref{SDEu}) if
\begin{equation}\label{equ-u}
\lim_{\e\downarrow 0}\frac{J(t,x,\hat u)-J(t,x, u^{\e})}{\e}\le 0
\end{equation}
for all $u\in \mathcal{U}[0,T]$, $x\in\dbR$ and $\ae t \in [0,T]$.
 The corresponding  equilibrium dynamics solves the SDE
\be\label{equ-SDE}\
\left\{\begin{array}{lll}
dX^{\hat u}(s)= b(s,X^{\hat u}(s),E[X^{\hat u}(s)],\hat u(s))ds+\sigma(s,X^{\hat u}(s),E[X^{\hat u}(s)],\hat u(s))dW(s),\, 0<s\le T,\\
X^{\hat u}(0)=x_0.
\end{array}
\right.
\ee
\end{definition}

 If feedback  strategies are to be used,  the previous definition reduces to the following

\begin{definition}\label{SGPE-feedback} A deterministic function  $\hat \varphi:\, [0,T]\times\dbR\longrightarrow U$ is a {\bf sub-game perfect Nash equilibrium}  for the system (\ref{J-1})-(\ref{SDEu}) if
\begin{equation}\label{equ-uf}
\lim_{\e\downarrow 0}\frac{J(t,x,\hat u)-J(t,x, u^{\e})}{\e}\le 0
\end{equation}
for all $u\in \mathcal{U}[0,T]$, $x\in\dbR$ and $\ae t \in [0,T]$, where $\hat u(s):=\hat \varphi(s,\bar X(s) )$, $0\le s\le T$ and $\bar X$ is given by \eqref{equ-SDE-feedback}.
 The associated equilibrium dynamics solves the SDE
\be\label{equ-SDE-feedback}
\left\{\begin{array}{lll}
d\bar X(s)= b(s,\bar X(s),E[\bar X^(s)],\hat \varphi(s,\bar X(s) ))ds+\sigma(s,\bar X(s),E[\bar X(s)],,\hat \varphi(s,\bar X(s) ))dW(s),\, 0<s\le T,\\
\bar X(0)=x_0.
\end{array}
\right.
\ee
\end{definition}

 For brevity, sometimes we  simply call $\hat u$ an equilibrium point when there is no ambiguity.

  The purpose of this study is to characterize sub-game perfect Nash equilibria for the system (\ref{SDEu})-(\ref{J-1}) by evaluating the limit (\ref{equ-u}) in terms of a stochastic maximum principle criterion. We will apply the general stochastic maximum principle for SDEs of mean-field type derived in Buckdahn {\it et al.} (2011) \cite{BDL}.

 The following assumptions (imposed in \cite{BDL}) will be in force throughout this paper. These assumptions can be made weaker, but we do not focus on this here.

 \bcd\label{cond1}
\no
\begin{itemize}
\item [$(i)$] The functions $b, \sigma, h,g$ are continuous in $(y,z,u)$, and bounded.
\item [$(ii)$] The functions $b, \sigma, h, g$ are twice continuously differentiable with respect to $(y,z)$, and their derivatives up to the second order are continuous in $(y,z, u)$, and bounded.
\end{itemize}
\ecd

Although we are interested in characterizing sub-game perfect Nash equilibrium points by considering the action of player $t$ at a deterministic position $x$, we perform the analysis for the more general case where player $t$ has a random variable $\xi\in L^2(\Omega,\cF_t, \dbP; \dbR)$ as a state.

 For a given admissible strategy $u\in \mathcal{U}[0,T]$, if player $t$ has $\xi\in L^2(\Omega,\cF_t, \dbP; \dbR)$ as state, (\ref{SDEu}) becomes
\begin{equation}\label{SDEu-flow}
\left\{\begin{array}{lll}
dX^{u,t,\xi}(s)= b(s,X^{u,t,\xi}(s),E[X^{u,t,\xi}(s)],u(s))ds+\sigma(s,X^{u,t,\xi}(s),E[X^{u,t,\xi}(s)],u(s))dW(s),\,\,  t< s\le T,\\
X^{u,t,\xi}(t)=\xi,
\end{array}
\right.
\end{equation}
and the associated  cost functional (\ref{J-1}) becomes
\begin{equation}\label{J-1-flow}
J(t,\xi,u)=E\left[\int_t^T h\l(s,X^{u,t,\xi}(s),E[X^{u,t,\xi}(s)],u(s)\r)
ds+g\l(X^{u,t,\xi}(T),E[X^{u,t,\xi}(T)]\r)\r].
\end{equation}

\begin{remark} \label{remark:Jxi}
 Definitions  \ref{SGPE} and \ref{SGPE-feedback} can be accordingly generalized by replacing $(t,x)$ by $(t,\xi)$ and the inequality condition takes the form
\begin{equation}\label{equ-u-xi}
\lim_{\e\downarrow 0}\frac{J(t,\xi,\hat u)-J(t,\xi, u^{\e})}{\e}\le 0
\end{equation}
for all $u\in {\cal U}[0,T]$, $\xi\in L^2(\Omega,  {\cal F}_t, {\dbP}; { \dbR})$ and a.e. $t\in [0,T]$.
  \end{remark}

 It is a well-known fact, see e.g. Karatzas and Shreve (\cite{KS}, pp. 289-290), that under Assumption (\ref{cond1}), for any $u\in \mathcal{U}[0,T]$,  the SDE (\ref{SDEu-flow}) admits a unique strong solution. Moreover, there exists a constant $C>0$ which depends only on the bounds of $b,\sigma$ and their first derivatives w.r.t. $y,z$, such that, for any $t\in[0,T], \, u\in \mathcal{U}[0,T]$ and $\xi, \xi^{\prime}\in L^2(\Omega,\cF_t, \dbP; \dbR)$, we also have the following estimates, $\dbP-\as$
 \be\label{estimate-flow}\begin{array}{lll}
 E[\sup_{t\le s\le T}|X^{u,t,\xi}(s)|^2|\cF_t]\le C(1+|\xi|^2+E[|\xi|^2]),\\
 E[\sup_{t\le s\le T}|X^{u,t,\xi}(s)-X^{u,t,\xi^{\prime}}(s)|^2|\cF_t]\le C(|\xi-\xi^{\prime}|^2+E[|\xi-\xi^{\prime}|^2]).
 \end{array}
 \ee
 Moreover, the performance functional (\ref{J-1-flow}) is well defined and finite.


 For convenience, we will use the following notation throughout the paper.
We will denote by
$X^{t,\xi}:=X^{u,t,\xi}$ the solution of the SDE (\ref{SDEu-flow}), associated with the strategy $u$, and correspondingly, $\hat X^{t,\xi}:=X^{\hat u,t,\xi}$ associated with $\hat u$.

 For $\varphi=b, \sigma, h, g$, we define
\be\label{notation-2}\left\{\begin{array}{llll}
\delta\varphi^{t,\xi}(s)=\varphi(s,\hat X^{t,\xi}(s),E[\hat X^{t,\xi}(s)],u(s))-\varphi(s,\hat X^{t,\xi}(s),E[\hat X^{t,\xi}(s)],\hat u(s)),\\
\varphi_y^{t,\xi}(s)=\frac{\partial\varphi}{\partial y}(s,\hat X^{t,\xi}(s),E[\hat X^{t,\xi}(s)],\hat u(s)),\quad \varphi^{t,\xi}_{yy}(s)=\frac{\partial^ 2\varphi}{\partial y^ 2}(s,\hat X^{t,\xi}(s),E[\hat X^{t,\xi}(s)],\hat u(s)),\\
\varphi^{t,\xi}_z(s)=\frac{\partial\varphi}{\partial z}(s,\hat X^{t,\xi}(s),E[\hat X^{t,\xi}(s)],\hat u(s)),\quad \varphi^{t,\xi}_{zz}(s)=\frac{\partial^ 2\varphi}{\partial z^ 2}(s,\hat X^{t,\xi}(s),E[\hat X^{t,\xi}(s)],\hat u(s)).
\end{array}\right.
\ee
 Let us introduce the Hamiltonian associated with the r.v. $X\in L^1(\Omega,\cF, \dbP)$:
\begin{equation}\label{hamiltonian}
H(s,X,u,p,q):=b(s,X,E[X],u)p+\sigma(s,X,E[X],u)q-h(s,X,E[X],u).
\end{equation}

\section{Adjoint equations and the stochastic maximum principle} \label{sec:smp}

 In this section we introduce the adjoint equations involved in the SMP which characterize the equilibrium points $\hat u\in \mathcal{U}[0,T]$ of our problem.

 The first order adjoint equation is the following linear backward SDE of mean-field type parametrized by $(t,\xi)\in [0,T]\times L^2(\Omega,\cF_t, \dbP; \dbR)$, satisfied by the processes $(p^{t,\xi}(s),q^{t,\xi}(s)),\, s\in[t,T],$
\be\label{pq-adjoint}
\left\{\begin{array}{lll}
dp^{t,\xi}(s)=-\{ H^{t,\xi}_y(s)+E\left[H^{t,\xi}_z(s)\right]\}ds+q^{t,\xi}(s)dW_s, \\  p^{t,x}(T)=-g^{t,\xi}_y(T)-E[g^{t,\xi}_z(T)],
\end{array}\right.
\ee
where, in view of the notation (\ref{notation-2}), for $j=y,z$,
\be\label{ddH}
H^{t,\xi}_{j}(s):=b^{t,\xi}_{j}(s)p^{t,\xi}(s)+\sigma^{t,\xi}_{j}(s)q^{t,\xi}(s)- h^{t,\xi}_{j}(s).
\ee
 This equation reduces to the standard one, when the coefficients do not explicitly depend on the expected value (or the marginal law) of the underlying diffusion process. Under Assumption \ref{cond1} on $b,\sigma, h, g$, by an adaptation of  Theorem 3.1. in Buckdahn, Li and Peng (2009), by keeping track of the parametrization $(t,\xi)$, the equation (\ref{pq-adjoint}) admits a unique $\dbF$-adapted solution $(p^{t,\xi},q^{t,\xi})$. Moreover,  there exists a constant $C>0$ such that, for all $t\in[0,T]$ and $\xi, \xi^{\prime}\in L^2(\Omega,\cF_t, \dbP; \dbR)$, we have the following estimate, $\Prob-a.s.,$
\be\label{pq-estimate}\begin{array}{lll}
E\l[\sup_{s\in[t,T]}|p^{t,\xi}(s)|^2+\int_t^T |q^{t,\xi}(s)|^2\, ds|\cF_t\r]\le  C(1+|\xi|^2+E[\xi^2]).
\end{array}
\ee

 The second order adjoint equation is the classical linear backward SDE, parametrized by $(t,\xi)\in [0,T]\times L^2(\Omega,\cF_t, \dbP; \dbR)$, that appears in Peng's stochastic maximum principle (see Peng (1990)):
\be\label{PQ-adjoint}
\left\{\begin{array}{lll}
dP^{t,\xi}(s)=-\(2b_y^{t,\xi}(s)P^{t,\xi}(s)+\left(\sigma_y^{t,\xi}(s)\right)^2P^{t,\xi}(s)+2\sigma_y^{t,\xi}(s)Q^{t,\xi}(s)+H^{t,\xi}_{yy}(s)\)\,ds+Q^{t,\xi}(s)\,dW_s, \\
P^{t,\xi}(T)=-g^{t,\xi}_{yy}(T),
 \end{array}\right.
\ee
where, in view of  (\ref{notation-2}),
\be\label{2-ddH}
 H^{t,\xi}_{yy}(s)=b^{t,\xi}_{yy}(s)p^{t,\xi}(s)+\sigma^{t,\xi}_{yy}(s)q^{t,\xi}(s)-h^{t,\xi}_{yy}(s).
\ee

\ms\no This is a standard linear backward SDE, whose unique $\dbF$-adapted solution $(P^{t,\xi},Q^{t,\xi})$ satisfies the following estimate:  There exists a constant $C>0$ such that, for all $t\in[0,T]$ and $\xi, \xi^{\prime}\in L^2(\Omega,\cF_t, \dbP; \dbR)$, we have,
\be\label{PQ-estimate}\begin{array}{lll}
E\l[\sup_{s\in[t,T]}|P^{t,\xi}(s)|^2+\int_t^T|Q^{t,\xi}(s)|^2\,ds|\cF_t\r]\le  C(1+|\xi|^2+E[\xi^2]),
\end{array}
\ee
$\dbP-\as$

The SDEs \eqref{pq-adjoint} and \eqref{PQ-adjoint} have a unique solution for a general control $u\in {\cal U}[0,T]$ and the corresponding estimates
\eqref{pq-estimate} and \eqref{PQ-estimate} hold. However, for Theorem \ref{MP1} below, only the equilibrium control $\hat u$ is substituted into the two equations.
 The following theorem is the first main  result of the paper.

\bt \label{MP1} $($Characterization of equilibrium strategies$)$ Let Assumption (\ref{cond1}) hold. Then $\hat u$ is an equilibrium strategy  for the  system (\ref{SDEu-flow})-(\ref{J-1-flow}) if and only if there are pairs of $\dbF$-adapted processes
$\left(p,q\right)$ and $\left(P,Q\right)$ which satisfy (\ref{pq-adjoint})-(\ref{pq-estimate}) and (\ref{PQ-adjoint})-(\ref{PQ-estimate}), respectively, and for which
\begin{equation}\label{SMP-2}
\begin{array}{ll}
   H(t,\xi,v,p^{t,\xi}(t),q^{t,\xi}(t))-H(t,\xi,\hat u(t),p^{t,\xi}(t),q^{t,\xi}(t))+\frac{1}{2}P^{t,\xi}(t)\left(\sigma(t,\xi, E[\xi],v)-\sigma(t,\xi, E[\xi],\hat u(t))\right)^2 \le 0, \\  \\  \qquad \qquad \,\,\,   \forall v\in U,\;
   \mbox{for all}\,\, v\in U,\; \xi\in L^2(\Omega,\cF_t, \dbP; \dbR), \; \ae t \in [0,T],\; \Prob-a.s.
\end{array}
\end{equation}
In particular, we have
\begin{equation}\label{SMP-2-x}
\begin{array}{ll}
   H(t,x,v,p^{t,x}(t),q^{t,x}(t))-H(t,x,\hat u(t),p^{t,x}(t),q^{t,x}(t))+\frac{1}{2}P^{t,x}(t)\left(\sigma(t,x, x,v)-\sigma(t,x, x,\hat u(t))\right)^2 \le 0, \\  \\  \qquad \qquad \,\,\,   \mbox{for all}\,\, v\in U,\; x\in\dbR, \; \ae t \in [0,T],\; \Prob-a.s.
\end{array}
\end{equation}
For feedback strategies, the deterministic function $\hat \varphi: \, [0,T]\times \dbR\longrightarrow U$ is an equilibrium strategy for the  system (\ref{J-1-flow})-(\ref{SDEu-flow}) if and only if there are pairs of $\dbF$-adapted processes
$\left(p,q\right)$ and $\left(P,Q\right)$ which satisfy (\ref{pq-adjoint})-(\ref{pq-estimate}) and (\ref{PQ-adjoint})-(\ref{PQ-estimate}), respectively, and for which
\begin{equation}\label{SMP-2-x-feedback}
\begin{array}{ll}
   H(t,x,v,p^{t,x}(t),q^{t,x}(t))-H(t,x,\hat \varphi(t,x),p^{t,x}(t),q^{t,x}(t))+\frac{1}{2}P^{t,x}(t)\left(\sigma(t,x, x,v)-\sigma(t,x, x,\hat\varphi(t,x))\right)^2 \le 0, \\  \\  \qquad \qquad \,\,\,   \mbox{for all}\,\, v\in U,\; x\in\dbR, \; \ae t \in [0,T],\; \Prob-a.s.
\end{array}
\end{equation}
\et

 \proof
Denote
\be
\delta H^{t,\xi}(s):= H(s,\hat X^{t,\xi}(s), u(s), p^{t,\xi}(s),q^{t,\xi}(s))-H(t,\hat X^{t,\xi}(s),\hat u(s),p^{t,\xi}(s),q^{t,\xi}(s))
\ee
where, the Hamiltonian $H$ is given by (\ref{hamiltonian}).
 \no  By  Theorem 2.1 in Buckdahn {\it et al.} (2011),  keeping track of the the parametrization $(t,\xi)$,  the key relation between the cost functional (\ref{J-1-flow}) and the associated Hamiltonian (\ref{hamiltonian}) reads
\begin{equation}\label{VI-2}
\begin{array}{lll}
J(t,\xi,\hat u)-J(t,\xi,u^{\e})=E\l[\int_t^{t+\e}\delta H^{t,\xi}(s)+\frac{1}{2}P^{t,\xi}(s)(\delta\sigma^{t,\xi}(s))^2\,ds\r] +R(\e),
\end{array}
\end{equation}
for arbitrary $u\in \mathcal{U}[0,T]$ and $(t,\xi)\in [0,T]\times L^2(\Omega,\cF_t, \dbP; \dbR)$,
where
$$
|R(\e)|\le \e\bar\rho(\e),
$$
for some function $\bar\rho: (0,\infty)\to (0,\infty)$ such that $\bar\rho(\e)\downarrow 0$ as $\e \downarrow 0$.

 Dividing both sides of (\ref{VI-2}) by $\e$ and then passing to the limit $\e\downarrow 0$, in view of Assumption \ref{cond1}, (\ref{pq-estimate}) and (\ref{PQ-estimate}),  we get
\be \label{VI-3}
\lim_{\e\downarrow 0}\frac{J(t,\xi,\hat u)-J(t,\xi, u^{\e})}{\e}=E\l[\delta H^{t,\xi}(t)+\frac{1}{2}P^{t,\xi}(t)(\delta\sigma^{t,\xi}(t))^2\r].
\ee
 Now, if (\ref{SMP-2}) holds, by setting $v:=u(t)$ for arbitrary $u\in \mathcal{U}[0,T]$,  we also get
\begin{equation*}
\begin{array}{ll}
   H(t,\xi,u(t),p^{t,\xi}(t),q^{t,\xi}(t))-H(t,\xi,\hat u(t),p^{t,\xi}(t),q^{t,\xi}(t))\\ \qquad\qquad \qquad\qquad \qquad +\frac{1}{2}P^{t,\xi}(t)\left(\sigma(t,\xi,E[\xi],u(t))-\sigma(t,\xi,E[\xi],\hat u(t))\right)^2 \le 0, \quad   \Prob-a.s.
\end{array}
\end{equation*}
Therefore, by (\ref{VI-3}) we obtain (\ref{equ-u-xi}), i.e. $\hat u$ is an equilibrium point for the  system (\ref{SDEu-flow})-(\ref{J-1-flow}).

 Conversely, assume that (\ref{equ-u-xi}) holds. Then, in view of (\ref{VI-3}), we have
\be\label{VI-4}
E\l[\delta H^{t,\xi}(t)+\frac{1}{2}P^{t,\xi}(t)(\delta\sigma^{t,\xi}(t))^2\r]\le 0,
\ee
for all $u\in \mathcal{U}[0,T]$, $\xi\in L^2(\Omega,\cF_t, \dbP; \dbR)$ and $\ae t \in [0,T]$. Now, let $A$ be an arbitrary set of ${\cal F}_t$ and set
$$
u(s):=v\ind_A+\hat u(s)\ind_{\Omega\setminus A}, \quad t\le s\le T,
$$
for an arbitrary $v\in U$. Obviously, $u$ is an admissible strategy. Moreover,  we have, for every $ s\in[t,T]$,
$$
\delta H^{t,\xi}(s)= \l(H(s,\hat X^{t,\xi}(s),v,p^{t,\xi}(s),q^{t,\xi}(s))-H(s,\hat X^{t,\xi}(s),\hat u(s),p^{t,\xi}(s),q^{t,\xi}(s)))\r)\ind_A,
$$
and
$$
\delta\sigma^{t,\xi}(s)=\l(\sigma(s,\hat X^{t,\xi}(s),E[\hat X^{t,\xi}(s)],v)-\sigma(s,\hat X^{t,\xi}(s),E[\hat X^{t,\xi}(s)],\hat u(s))\r)\ind_A.
$$
Hence, in view of (\ref{VI-4}), we have
\begin{equation*}
\begin{array}{lll}
E\big[\big(H(t,\hat X^{t,\xi}(t),v,p^{t,\xi}(t),q^{t,\xi}(t))-H(t, \hat X^{t,\xi}(t),\hat u(t),p^{t,\xi}(t),q^{t,\xi}(t))\big)\ind_A\big]\\ \\ +\frac{1}{2}E\big[ P^{t,\xi}(t)\big(\sigma(t,\hat X^{t,\xi}(t),E[\hat X^{t,\xi}(t)],v)-\sigma(t,\hat X^{t,\xi}(t),E[\hat X^{t,\xi}(t)],\hat u(t))\big)^2\ind_A \big]\le 0,
\end{array}
\end{equation*}
which in turn yields the inequality (\ref{SMP-2})  since  $v\in U$  and the set $A\in {\cal F}_t$ are arbitrary.

 Finally, both (\ref{SMP-2-x}) and (\ref{SMP-2-x-feedback}) follow from (\ref{SMP-2}), by replacing $\xi\in L^2(\Omega,\cF_t, \dbP; \dbR)$ with $x\in \dbR$. \qed

\begin{remark}\label{free-smp} Define the so-called $\cal H$-function associated with $(\hat u(t),p^{t,\xi}(t),q^{t,\xi}(t), P^{t,\xi}(t))$
\be\label{H-1}\begin{array}{lll}
{\cal H}(t,\xi,v):=H(t,\xi,v,p^{t,\xi}(t),q^{t,\xi}(t))-\frac{1}{2}P^{t,\xi}(t)\sigma^2(t,\xi,E[\xi],\hat u(t))\\ \q\qq\qq +
\frac{1}{2}P^{t,\xi}(t)\l(\sigma(t,\xi,E[\xi],v)-\sigma(t,\xi,E[\xi],\hat u(t))\r)^2. \nonumber
\end{array}
\ee
Then,  it is easily checked that the inequality (\ref{SMP-2}) is equivalent to
\be\label{H-2}
{\cal H}(t,\xi,\hat u(t))=\max_{v\in U}{\cal H}(t,\xi,v), \;\;\; \mbox{for all}\,\, \xi\in L^2(\Omega,\cF_t, \dbP; \dbR), \; \ae t \in [0,T],\; \Prob-a.s.
\ee
\end{remark}

 For all practical purposes, it would be nice to find or characterize equilibrium points, through only maximizing the Hamiltonian $H$, which amounts to only solving the first order adjoint equation (\ref{pq-adjoint}). In fact, this happens in the special case where the diffusion coefficient does not contain the control variable, i.e.,
$$
\sigma(s,y,z,v)\equiv \sigma(s,y,z),\q (s,y,z,v)\in [0,T]\times\dbR\times\dbR\times U,
$$
whence, manifestly, the inequality (\ref{SMP-2}) is equivalent to
\be\label{H-3}
H(t,\xi,\hat u(t),p^{t,\xi}(t),q^{t,\xi}(t))=\underset{v\in U}{\max}\, H(t,\xi,v,p^{t,\xi}(t),q^{t,\xi}(t)), \nonumber
\ee
for all $\xi\in L^2(\Omega,\cF_t, \dbP; \dbR), \; \ae t \in [0,T],\; \Prob-a.s.$

 Another very useful case, which we will use in some examples below, is described in the following
\begin{proposition}\label{useful}
Assume that  $U$ is a convex subset of $\dbR$, and the coefficients $b, \sigma$ and $h$ satisfy the assumption \ref{cond1}, and are such that $H(t,y,\cdot,p,q)$ is concave for all $(t,y)\in[0,T]\times \dbR$ almost surely. Then, the admissible strategy $\hat u$ is an equilibrium point for the  system (\ref{SDEu-flow})-(\ref{J-1-flow}) if and only if there is a pair of $\dbF$-adapted processes
$\left(p^{t,\xi},q^{t,\xi}\right)$  that satisfies (\ref{pq-adjoint})-(\ref{pq-estimate}) and for which
\begin{equation}\label{SMP-3}
\begin{array}{ll}
   H(t,\xi,\hat u(t),p^{t,\xi}(t),q^{t,\xi}(t))=\underset{v\in U}{\max}\, H(t,\xi,v,p^{t,\xi}(t),q^{t,\xi}(t)),
\end{array} \nonumber
\end{equation}
for all $\xi\in L^2(\Omega,\cF_t, \dbP; \dbR),\; \ae t \in [0,T],\; \Prob-a.s.$
\end{proposition}

\begin{proof} In view of (\ref{H-2}) it suffices to show that $\cal H$ and $H$ have the same Clark's generalized gradient in $\hat u$. But, this follows e.g. from Lemma 5.1. in Yong and Zhou (1999), since $U$ is a convex subset of $\dbR$ and  $H(t,y,\cdot,p,q)$ is concave for all $(t,y)\in[0,T]\times \dbR$ almost surely. Hence, $\hat u$ is a maximizer of ${\cal H}(t,\xi,\cdot,p^{t,\xi}(t),q^{t,\xi}(t))$ if and only if it is a maximizer of $H(t,\xi,\cdot,p^{t,\xi}(t),q^{t,\xi}(t))$.
\end{proof}

\begin{remark} In fact both Theorem \ref{MP1} and Proposition \ref{useful} extend to the following cost functionals parametrized by $(t,\xi)\in [0,T]\times L^2(\Omega,\cF_t, \dbP; \dbR):$
\begin{equation}\label{J-2}
J(t,\xi,u)=E\left[\int_t^T h\l(t,\xi,s,X^{t,\xi}(s),E[X^{t,\xi}(s)],u(s)\r)ds
+g\l(t,\xi,X^{t,\xi}(T),E[X^{t,\xi}(T)]\r)\r], \nonumber
\end{equation}
where both $h$ and $g$ are allowed to explicitly depend on $(t,x)$. This is due to the fact that the spike variation and the subsequent Taylor expansions that are used to derive (\ref{VI-2}) are not affected by this extra dependence of $h$ and $g$ on $(t,\xi)$.
\end{remark}

\section{Some applications}
\label{sec:app}

In this section we illustrate the above results through some examples discussed in Bj\"ork and Murgoci (2008) and  Bj\"ork, Murgoci and Zhou (2011), using an extended Hamilton-Jacobi-Bellman equation. In these examples, we look for equilibrium strategies of {\it  feedback-type} i.e. deterministic function $\hat\varphi: [0,T]\times \dbR\longrightarrow U$ which satisfy (\ref{SMP-2-x-feedback}). The corresponding equilibrium point is $\hat u(s):=\hat\varphi(s,\hat X(s))$, where, $\hat X$ is corresponding to the equilibrium dynamics  given by the SDE
$$
\left\{\begin{array}{lll}
d\hat X(s)=b(s,\hat X(s),E[\hat X(s)],\varphi(s,\hat X(s)))ds+\sigma(s,\hat X(s),E[\hat X(s)],\varphi(s,\hat X(s)))dW(s),\,\, 0<s\le T,
\\
\hat X(0)=x_0.
\end{array}
\right.
$$

\subsection{ Mean-variance portfolio selection with constant risk aversion}
\label{sec:sub:mv}

The dynamics over $[0,T]$  defined on $(\Omega,\cF,\dbF,\dbP)$ is  given by the following  SDE:
\begin{equation}\label{geo-2}
 dX(s)=\l(r X(s)+\l(\alpha-r\r)u(s)\r)ds+\sigma u(s)dW(s),\qquad X(0)=x_0 \, ( \in \dbR),
\end{equation}
where $r, \alpha$ and $\sigma$ are real constants.

  The cost functional  is given by
\begin{equation}\label{geo-J}\begin{array}{lll}
 J(t,x,u)=\frac{\gamma}{2}Var(X^{t,x}(T))-E[X^{t,x}(T)] \\
\qq\qq = E\l(\frac{\gamma}{2}\l(X^{t,x}(T)\r)^2-X^{t,x}(T) \r)-\frac{\gamma}{2}\l(E[X^{t,x}(T)]\r)^2,
\end{array}
\end{equation}
where the constant $\gamma$, assumed positive, is the risk aversion coefficient. The associate dynamics, parametrized by $(t,x)\in[0,T]\times\dbR$ is

\begin{equation}\label{geo-2-u-t}
 dX^{t,x}(s)=\l(r X^{t,x}(s)+\l(\alpha-r\r)u(s)\r)ds+\sigma u(s)dW(s),\,\, t<s\le T, \qquad X^{t,x}(t)=x.
\end{equation}

 The Hamiltonian associated to this system is
\begin{equation*}
 H(t,x,u,p,q)=\l(r x+\l(\alpha-r\r)u\r)p+\sigma uq,
\end{equation*}
and the $\cal H$-function is
\be\label{H-1}\begin{array}{lll}
{\cal H}(t,x,v):=H(t,x,v,p,q)-\frac{1}{2}P(\sigma \hat \varphi(t,x))^2+
\frac{1}{2}P\sigma^2\l(v-\hat \varphi(t,x)\r)^2. \nonumber
\end{array}
\ee

 The equation for $P$ takes the form
\begin{equation}
dP^{t,x}(s)= -2r P^{t,x}(s)ds +Q^{t,x}(s)dW_s,
\end{equation}
where $P^{t,x}(T)= -\gamma$. We obtain $P^{t,x}(s)= -\gamma e^{2r(T-s)}$ for $s\in [t, T]$.

 In view of Remark \ref{free-smp}, $\hat\varphi$ is a equilibrium point if and only if it maximizes the $\cal H$-function. Such a maximum  exists if and only if
\begin{equation}\label{E-2-2}
(\alpha-r)p+\sigma q=0.
\end{equation}
Therefore, to characterize the equilibrium points, we only need to consider the first-order adjoint equation:
\begin{equation}
\left\{
\begin{array}{lll}
 dp^{t,x}(s)=-r p^{t,x}(s)ds+q^{t,x}(s)dW(s), \label{appl-2} \\ \\
p^{t,x}(T)=1-\gamma\l(\hat X^{t,x}(T)-E[\hat X^{t,x}(T)]\r).
\end{array}
\right.
\end{equation}
We try a solution of the form
\begin{equation}\label{Ansats}
p^{t,x}(s)=C_s-A_s\l(\hat X^{t,x}(s)-E[\hat X^{t,x}(s)]\r),
\end{equation}
where $A_s$ and $C_s$ are deterministic functions such that
$$
A_T=\gamma,\,\,\, C_T=1.
$$
Identifying the coefficients in (\ref{geo-2-u-t}) and (\ref{appl-2}), we get, for $s\ge t$,
\begin{equation}\label{appl3}
\begin{array}{lll}
(2r A_s+\dot A_s)\l(\hat X^{t,x}(s)-E[\hat X^{t,x}(s)]\r)+(\alpha-r)A_s(\hat\varphi(s,\hat X^{t,x}(s))-E[\hat\varphi(s,\hat X^{t,x}(s))])=\dot C_s+r C_s,
\end{array}
\end{equation}
\begin{equation}
q^{t,x}(s)=-A_s\sigma \hat\varphi(s,\hat X^{t,x}(s)).\label{appl4}
\end{equation}
In view of (\ref{E-2-2}), we have
\begin{equation}\label{E-2-2-1}
(\alpha-r)p^{t,x}(t)+\sigma q^{t,x}(t)=0.
\end{equation}
Now, from (\ref{Ansats}), we have
\begin{equation*}
p^{t,x}(t)=C_t,
\end{equation*}
which is deterministic and independent of $x$.
\no Hence, from (\ref{E-2-2}) we get
\begin{equation}\label{E-2-3}
 q^{t,x}(t)=-\frac{\alpha-r}{\sigma}C_t. \nonumber
\end{equation}
In view of (\ref{appl4}), the equilibrium point is the deterministic function
\begin{equation}\label{u-2}
 \hat\varphi(s):=\frac{\alpha-r}{\sigma^2}\frac{C_s}{A_s}, \quad 0\le s\le T.
\end{equation}
It remains to determine  $A_s$ and $C_s$.

 Indeed, inserting (\ref{u-2}) in (\ref{appl3}) we obtain
\begin{equation}\label{appl5}
\begin{array}{lll}
(\dot A_s+2r A_s)(\hat X(s)-E[\hat X^{t,x}(s)])=\dot C_s+r C_s,
\end{array} \nonumber
\end{equation}
giving the equations satisfied by $A_s$ and $C_s$
\begin{equation*}
 \left\{ \begin{array}{lll}
\dot A_s+2r A_s=0,&~ A_T=\gamma,\\
\dot C_s+r C_s=0,&~ C_T=1.
\end{array}
\right.
\end{equation*}
The solutions of these equations are
\begin{equation}
A_s=\gamma e^{2r (T-s)},\quad C_s=e^{r (T-s)},\quad 0\le s\le T. \nonumber
\end{equation}
Whence, we obtain the following explicit form of the equilibrium point:
\begin{equation}\label{u-3-2}
 \hat\varphi(s)=\frac{1}{\gamma}\frac{\alpha-r}{\sigma^2}e^{-r (T-s)},\qq 0\le s\le T, \nonumber
\end{equation}
which is identical to the one obtained in Bj\"ork and Murgoci (2008) by solving an extended HJB equation.

\subsection{Mean-variance portfolio selection with state dependent risk aversion}
Consider the same state process over $[0,T]$ as in Section \ref{sec:sub:mv}. Namely,
\begin{equation}\label{s-geo-1}
 dX(s)=\l(r X(s)+\l(\alpha-r\r)u(s)\r)ds+\sigma u(s)dW(s),\qquad X(0)=x_0,
\end{equation}
where $r, \alpha$ and $\sigma$ are real constants.  The modified cost functional takes the form
\begin{equation*}\begin{array}{lll}
 J(t,x,u)=\frac{\gamma(x)}{2}Var(X^{t,x}(T))-E[X^{t,x}(T)],  \\
\end{array}
\end{equation*}
where the risk aversion coefficient $\gamma(x)$ is made dependent on the current wealth $x$. We refere to Bj\"ork  {\it et al.} (2011) for an economic motivation of this dependence.

 The associated dynamics, parametrized by $(t,x)\in[0,T]\times\dbR$ is
\begin{equation}\label{s-geo-2}
 dX^{t,x}(s)=\l(r X^{t,x}(s)+\l(\alpha-r\r)u(s)\r)ds+\sigma u(s)dW(s),\,\, t<s\le T, \qquad X^{t,x}(t)=x.
\end{equation}

  Now, since $\gamma(x)$ is assumed strictly positive for all $x$, the equilibrium points of $J$ are the same as the ones of the the cost functional
\be \label{s-J-1}
 \bar J(t,x,u)=\frac{1}{2}Var(X^{t,x}(T))- \gamma^{-1}(x)E[X^{t,x}(T)].
\ee
Therefore, we will find feedback equilibrium points  associated with (\ref{s-J-1}).

The Hamiltonian associated to this system  is
\begin{equation*}
 H(t,x,u,p,q)=\l(r x+\l(\alpha-r\r)u\r)p+\sigma uq.
\end{equation*}
and the $\cal H$-function is
\be\label{H-1}\begin{array}{lll}
{\cal H}(t,x,v):=H(t,x,v,p,q)-\frac{1}{2}P(\sigma\hat \varphi(t,x))^2+
\frac{1}{2}P\sigma^2\l( v-\hat \varphi(t,x)\r)^2. \nonumber
\end{array}
\ee
Again, in view of Remark \ref{free-smp}, $\hat\varphi$ is a equilibrium point if and only if it maximizes the $\cal H$-function. Such a maximum  exists if and only if
\begin{equation}\label{E-2-3}
(\alpha-r)p+\sigma q=0.
\end{equation}
Therefore, to characterize the equilibrium points, we only need to consider the first-order adjoint equation:
\begin{equation}
\left\{
\begin{array}{lll}
 dp^{t,x}(s)=-r p^{t,x}(s)ds+q^{t,x}(s)dW(s), \label{s-pq} \\ \\
p^{t,x}(T)=\gamma^{-1}(x)-\l(\hat X^{t,x}(T)-E[\hat X^{t,x}(T)]\r),
\end{array}
\right.
\end{equation}
We try a solution of the form
\begin{equation}\label{s-Ansats}
p^{t,x}(s)=C_s\gamma^{-1}(x)-A_s\l(\hat X^{t,x}(s)-E[\hat X^{t,x}(s)]\r),
\end{equation}
where $A_s, B_s$ and $C_s$ are deterministic functions such that
$$
A_T=C_T=1.
$$
Identifying the coefficients in (\ref{s-geo-2}) and (\ref{s-pq}), we get for $s\ge t$,
\begin{equation}\label{s-appl3}
\begin{array}{lll}
(\dot A_s+2r A_s)\l(\hat X^{t,x}(s)-E[\hat X^{t,x}(s)]\r)+(\alpha-r)A_s\l( \hat\varphi(s,\hat X^{t,x}(s))-E[ \hat\varphi(s,\hat X^{t,x}(s))]\r)\\ \qquad\qquad\qquad\qquad\qquad\qquad\qquad\qquad\qquad =(\dot C_s+r C_s)\gamma^{-1}(x),
\end{array}
\end{equation}
\begin{equation}
q^{t,x}(s)=-A_s\sigma \varphi(s,\hat X^{t,x}(s)),\label{s-appl4}
\end{equation}
and, by (\ref{E-2-3}), we have
\begin{equation}\label{s-E-2-2}
(\alpha-r)p^{t,x}(t)+\sigma q^{t,x}(t)=0,
\end{equation}
But, from (\ref{s-Ansats}) we have
\begin{equation}
p^{t,x}(t)=C_t\gamma^{-1}(x). \nonumber
\end{equation}
Therefore, we get from (\ref{s-E-2-2})
\begin{equation}\label{s-E-2-3}
 q^{t,x}(t)=-\frac{\alpha-r}{\sigma}C_t\gamma^{-1}(x),
\end{equation}
which together with  (\ref{s-appl4}), suggest that an equilibrium point $ \hat\varphi$ of the form
\begin{equation}\label{s-u-2}
 \hat\varphi(s,y)=\frac{\alpha-r}{\sigma^2}\frac{C_s}{A_s}\gamma^{-1}(y),\qquad 0\le s\le T.
\end{equation}
It remains to determine  $A_s$ and $C_s$.

 Indeed, inserting (\ref{s-u-2}) in (\ref{s-appl3}) we obtain,
\be\label{s-appl5}
\begin{array}{lll}
(\dot A_s+2r A_s)\l(\hat X^{t,x}(s)-E[\hat X^{t,x}(s)]\r)+\frac{(\a-r)^2}{\sigma^2}C_s\l(\gamma^{-1}(\hat X^{t,x}(s))-E[\gamma^{-1}(\hat X^{t,x}(s))]\r)\\ \qquad\qquad\qquad\qquad\qquad\qquad\qquad\qquad\qquad =(\dot C_s+r C_s)\gamma^{-1}(x).
\end{array}
\ee

 Manifestly, from (\ref{s-appl5}) it is hard to draw any conclusion about the form of the deterministic functions $A_s$ and $C_s$ unless we have an explicit form of the function $\gamma(x)$. In fact, a closer look at (\ref{s-appl5}) suggests that a feasible identification of the coefficients is possible, for instance, when  $\gamma(x)=\gamma/x$. Let us examine this case.

\subsubsection*{The case $\gamma(x)=\frac{\gamma}{x}$}

 Let us consider the particular case when
$$
\gamma(x)=\frac{\gamma}{x}.
$$
In this special case,  (\ref{s-appl5}) becomes
\be\label{s-appl6}
\begin{array}{lll}
(\dot A_s+2rA_s+\frac{(\alpha-r)^2}{\gamma\sigma^2}C_s)\l(X^{t,x}(s)-E[X^{t,x}(s)]\r)-(\dot C_s+r C_s)\frac{x}{\gamma}=0.
\end{array}  \nonumber
\ee
This suggests that the functions $A_s, B_s$ and $C_s$ solve the following system of equations:
\be\left\{\begin{array}{lll}
\dot A_s+2rA_s+\frac{(\alpha-r)^2}{\gamma\sigma^2}C_s=0,
 \\
\dot C_s+r C_s=0, \label{s-appl7}\\
A_T=C_T=1,
\end{array}
\right.
\ee
which admits the following explicit solution:
\be
A_s=e^{2r (T-s)}+\frac{(\alpha-r)^2}{r\gamma\sigma^2}\l(e^{2r(T-s)}-e^{r(T-s)}\r),\quad C_s=e^{r(T-s)},\quad 0\le s\le T. \nonumber
\ee
Hence, the equilibrium point  $\hat\varphi$ explicitly given by
\begin{equation}\label{s-u-3}
\begin{array}{lll}
 \hat\varphi(s,y)=\frac{\alpha-r}{\gamma\sigma^2}\frac{C_s}{A_s}y\\ \qquad\quad =\frac{\alpha-r}
{\gamma\sigma^2}\l(e^{r (T-s)}+\frac{(\alpha-r)^2}{r\gamma\sigma^2}
\l(e^{r(T-s)}-1\r)\r)^{-1}y,\qquad
  (s,y)\in [0,T]\times\dbR.
 \end{array} \nonumber
\end{equation}

\subsection{Time-inconsistent linear-quadratic regulator}\label{sec:lqr}

We consider the following variant of a time-inconsistent  linear-quadratic regulator  discussed in Bj\"ork and Murgoci (2008). We refer to recent work by Bensoussan {\it et al.} (2013), Yong (2013a), and Hu {\it et al.} (2012), where more general models are considered. The state process over $[0,T]$ defined on $(\Omega,\cF,\dbF,\dbP)$  is a scalar with dynamics
\begin{equation}\label{LGMF}
 dX(s)=\l(aX(s)+bu(s)\r)ds+\sigma dW(s),\quad X(0)=x_0,
\end{equation}
where $ a, b$ and $\sigma$ are real constants.
 The cost functional is given by
\begin{equation}\label{J-LQ}
 J(t,x,u)=\frac{1}{2}E\l[\int_t^T u^2(s)\,ds\r]+\frac{\gamma}{2}E\left[ \left(X^{t,x}(T)-x\right)^2\right],  \nonumber
\end{equation}
where $\gamma$ is a positive constant. The associated dynamics, parametrized by $(t,x)\in[0,T]\times\dbR$ is
\begin{equation}\label{LGMF-t}
 dX^{t,x}(s)=\l(aX^{t,x}(s)+bu(s)\r)ds+\sigma dW(s),\,\, t<s\le T, \qquad X^{t,x}(t)=x.
\end{equation}

  As mentioned in Bj\"ork and Murgoci (2008), in this time-inconsistent version of the linear-quadratic regulator, we want to control the system so that the final state $X^{t,x}(T)$ stays as close as possible to $X^{t,x}(t)=x$, while at the same time we keep the control energy (expressed by the integral term) small. The time-inconsistency stems from the fact that the target point $X^{t,x}(t)=x$ is changing with time.

  The Hamiltonian associated to this system is
\begin{equation}\label{H(t)}\begin{array}{ll}
H(s,x,u,p,q):=\l(ax+bu\r)p+\sigma q-\frac{1}{2}u^2.
\end{array}
\end{equation}
and the $\cal H$-function is
\be\label{H-1}\begin{array}{lll}
{\cal H}(t,x,v):=H(t,x,v,p,q)-\frac{1}{2}P\sigma^2 \nonumber
\end{array}
\ee
Again, in view of Remark \ref{free-smp}, $\hat\varphi$ is a equilibrium point if and only if it maximizes the $\cal H$-function. Such a maximizer  is
\begin{equation}\label{u-1}
\hat \varphi=bp.
\end{equation}
Therefore, to characterize the equilibrium points, we only need to consider the first-order adjoint equation:
\begin{equation}
\left\{
\begin{array}{lll}
 dp^{t,x}(s)=-ap^{t,x}(s)ds+q^{t,x}(s)dW(s), \label{p-LGMF} \\ \\
p^{t,x}(T)=\gamma(x-X^{t,x}(T)).
\end{array}
\right.
\end{equation}
We try a solution of the form
\be\label{Ansats-2}
p^{t,x}(s)=\b_sx-\a_s\hat X^{t,x}(s),
\ee
where $\a_s$ and $\b_s$ are deterministic functions such that
\be
\a_T=\b_T=\gamma. \nonumber
\ee
Identifying the coefficients in (\ref{LGMF-t}) and (\ref{p-LGMF}), we get, for $s\ge t$,
\be\label{a-beta}
(\dot\a_s+2a\a_s)\hat X^{t,x}(s)+b\a_s\hat \varphi(s,\hat X^{t,x}(s))=(\dot\b_s+a\b_s)x,
\ee
and
\be
q^{t,x}(s)=-\sigma\a_s. \nonumber
\ee
On the other hand, in view of (\ref{u-1})
\begin{equation}\label{u-hat-1}
\hat \varphi(t,x)=bp^{t,x}(t). \nonumber
\end{equation}
Thus, by (\ref{Ansats-2}), the function $\varphi$ which yields the equilibrium point has the form
\begin{equation}\label{u-hat-2}
\hat\varphi(s,y)=b(\b_s-\a_s)y, \quad (s,y)\in[0,T]\times \dbR.
\end{equation}
Therefore, (\ref{a-beta}) reduces to
\be
(\dot\a_s+(2a+b^2\b_s)\a_s-b^2\a^2_s)\hat X^{t,x}(s)=(\dot\b_s+a\b_s)x, \nonumber
\ee
suggesting that $(\a_s,\b_s)$ solves the system of equations
\begin{equation}\label{beta-1}
 \left\{ \begin{array}{lll}
\dot\b_s+a\b_s=0,\\
\dot\a_s+(2a+b^2\b_s)\a_s-b^2\a^2_s=0,\\
\a_T=\gamma, \,\,\b_T=\gamma.
\end{array}
\right.
\end{equation}

  The first equation in \eqref{beta-1} yields  the solution
$$
\b_s=\gamma e^{a(T-s)}.
$$
The second equation is of Riccati type whose solution is $\a_s:=\frac{v_s}{w_s}$, where $(v,w)$ solves the following system of linear differential equation:
$$
\left(\begin{array}{lll}\dot  v_s \\ \dot  w_s\end{array}\right)=\left(\begin{array}{lll} -2a & 0 \\ -b^2& b^2\b_s\end{array}\right)\left(\begin{array}{lll} v_s \\ w_s\end{array}\right),\quad \left(\begin{array}{lll} v_T\\ w_T\end{array}\right)=\left(\begin{array}{lll} \gamma \\ 1 \end{array}\right).
$$
which is obviously solvable.


\section{Extension to mean-field game models}
\label{sec:game}

In this section we extend the SMP approach to
 an $N$-player stochastic differential game of mean-field type where the $i$-th player  would like to find a strategy to optimize her own cost functional regardless of the other players' cost functionals.

Let  $X=(X_1, \ldots, X_N)$ describe the states of the $N$ players and  $u=(u_1, \ldots, u_N)\in\Pi_{i=1}^N\cU_i[0,T]$ be the ensemble of all the individual admissible strategies. Each $u_i$ takes values   in a non-empty subset $U_i$ of $\dbR$  and the class of admissible strategies is given by
\begin{equation}
\cU_i[0,T]=\Big\{u_i: [0,T]\times \Omega \longrightarrow U_i;\,\,   u_i\, \mbox{is}\ \dbF\mbox{-adapted and square integrable}\Big\}.
\end{equation}
To simplify the analysis, we consider a population of uniform agents so that $U_i=U$
and they have the same initial state $X_i(0)=x_0$ at time 0  for all $i\in \{1, \ldots, N\}$. In this case, the $N$ sets $\cU_i[0,T]$ are identical and equal to $\cU[0,T]$ .
Let the dynamics be given by the following SDE:
\begin{equation}\label{N-SDEu}
dX_i(s)= b(s,X_i(s),E[X_i(s)],u_i(s))ds+\sigma(s,X_i(s),E[X_i(s)])dW_i(s),
\end{equation}
where the strategy $u_i$ does not enter the diffusion coefficient  $\sigma$.

 For notational simplicity, we do not explicitly indicate the dependence of the state on the control by writing $X_i^{u_i}(s)$. We take $\mathbb{F}$ to be the natural  filtration  of the $N$-dimensional standard Brownian motion $(W_1, \ldots, W_N)$ augmented by $\mathbb{P}$-null sets of ${\cal F}$.

 Denote
$$
(u_{-i},v):=(u_1,\ldots,u_{i-1},v,u_{i+1},\ldots,u_N),\quad i=1,\ldots,N.
$$
Then, the $i$-th  player selects $u_i\in\cU[0,T]$ to evaluate her cost functional
$$
v\mapsto J^{i,N}(t,x_i;u_{-i},v):=J^{i,N}(t,x_i;u_1,\ldots,u_{i-1},v,u_{i+1},\ldots,u_N),
$$
 where,  for $i\in\{1,\ldots,N\}$,
\be\label{N-J-1}
\begin{array}{lll}
J^{i,N}(t,x_i;u)=E\left[\int_t^T h\l(s,X^{t,x_i}_i(s),E[X^{t,x_i}_i(s)], X^{(-i)}(s),u_i(s)\r)ds \right.
\\ \qquad\qquad\qquad\qquad\qquad\qquad \qquad\qquad + \left.g\l(X^{t,x_i}_i(T),  E[X^{t,x_i}_i(T)], X^{(-i)}(T) \r)\r],
\end{array}
\ee
whose associated dynamics, parametrized by $(t,x_i)$, is
\begin{equation}\label{N-SDEu-flow}\left\{\begin{array}{lll}
dX^{t,x_i}_i(s)= b(s,X^{t,x_i}_i(s),E[X^{t,x_i}_i(s)],u_i(s))ds+\sigma(s,X^{t,x_i}_i(s),E[X^{t,x_i}_i(s)])dW_i(s),\,\, t<s\le T,
\\
X^{t,x_i}(t)=x_i.
\end{array}
\right.
\end{equation}

The $i$-th player interacts with others  through the mean-field  coupling term
$$
X^{(-i)}=\frac{1}{N-1}\sum_{k\ne i}^N X_k,\qquad i\in\{1,\ldots,N\},
$$ which models the aggregate impact of all other players.

 Note that the $i$-th player assesses her cost functional over $[t,T]$ seen from  her local state $X_i(t)= x_i$
and she knows only the initial states of all other players at time 0, $(X_k(0)=x_0$, $ k\neq i$).   Thus the game may be cast as a decision problem where each player has incomplete state information about other players. The development of a solution framework in terms of a certain exact equilibrium notion is challenging. Our objective is to address this incomplete state information issue and design a set of individual strategies which has a meaningful interpretation. This will be achieved by using the so-called {\it  consistent mean-field approximation}.

 For a large  $N$, even if each player has full state information of the system,  the exact characterization of the equilibrium points, based on the SMP, will have high complexity since one needs to solve a very high dimensional system of coupled variational inequalities for the underlying Hamiltonians similar to (\ref{SMP-2}).
Therefore, we should rely on the mean-field approximation of our system.

 We note that $J^{i,N}$  depends on not only $u_i$, but also all other players' strategies $u_{-i}$ through the mean-field coupling term $X^{(-i)}$.  This suggests that we extend Definition \ref{SGPE} to the $N$-player case as follows.

\begin{definition}\label{N-SGPE} The admissible strategy $\hat u=(\hat u_1, \ldots, \hat u_N)$ is an asymptotic (in population size $N$) sub-game perfect Nash equilibrium point for the system (\ref{N-SDEu})-(\ref{N-J-1}) if for every $i\in\{1,\ldots,N\}$,
\begin{equation}\label{equ-u-N}
\lim_{\e\downarrow 0}\frac{J^{i,N}(t,x_i;\hat u)-J^{i,N}(t,x_i;\hat u_{-i}, u_i^{\e})}{\e}\le O(\delta_N),
\end{equation}
for each given $u_i\in \cU_i[0,T],\,x_i\in\dbR$ and $\, \ae t\in[0,T]$, where $u_i^{\e}$ is the spike variation (\ref{spike}) of the  strategy $\hat u_i$ of the $i$-th player using $u_i$ and $0\le \delta_N\rightarrow 0 $ as $N\rightarrow \infty$.
\end{definition}

 The error term $O(\delta_N)$ is due to the mean field approximation to be introduced below for designing $\hat u$.

\subsection{The local limiting decision problem}

Let $X^{(-i)}$ be approximated by a deterministic  function
$\bar X(s)$ on $[0,T]$. Denote
the cost functional
\be\label{J-i}\begin{array}{lll}
\bar J^{i} (t, x_i; u_i) =  E\left[\int_t^T h\left(s, X^{t,x_i}_i(s),  E[ X^{t,x_i}_i(s) ], \bar X(s), u_i(s)\right)ds   \right.
\\ \qquad\qquad\qquad\qquad\qquad\qquad \qquad\qquad \qquad\qquad \left. +g\left(X^{t,x_i}_i(T), E[ X^{t,x_i}_i(T)], \bar X(T)\right)\right]
\end{array}
\ee
which is intended as an approximation of $J^{i,N}$.
Note that once $\bar X$ is assumed fixed, $\bar J^{i}$ is affected only by $u_i$.
 The introduction of $\bar X$ as a {\it fixed} function of time is based on the freezing idea in mean field games. The reason is that  $X^{(-i)}=\frac{1}{N-1}
\sum_{k=1}^N X_k$ is generated by many  negligibly small players, and therefore a given player has little influence on it.

The strategy selection of the $i$-th player  is based on finding a sub-game perfect Nash equilibrium  for $\bar J^{i}$ to which the method based on the Stochastic Maximum Principle (cf.\@ (\ref{SMP-2-x})) of Section \ref{sec:smp} can be applied under the following conditions:

 \bcd\label{cond3}
\no
\begin{itemize}
\item [$(i)$] The functions $b, \sigma, h,g$ are  continuous in $(y,z,u)$, and bounded.
\item [$(ii)$] The functions $b, \sigma$ are twice continuously differentiable with respect to $(y,z)$, and their derivatives up to the second order are continuous in $(y,z, u)$, and bounded.
\item [$(iii)$] The functions $ h, g$ are twice continuously differentiable with respect to $(y,z,w)$, and their derivatives up to the second order are continuous in $(x,y, w, u)$, and bounded.
\end{itemize}
\ecd

Let $\hat u_i\in \cU [0,T]$ be a  sub-game perfect Nash equilibrium point for (\ref{N-J-1})-(\ref{N-SDEu-flow}) and denote the  associated  backward SDE
\begin{equation} \label{ppH}
\left\{
\begin{array}{ll}
dp^{t,x_i}(s)=-\{ H^{t,x_i}_y(s)+E [H^{t,x_i}_z(s)]\} ds +q^{t,x_i}(s)dW_i(s),\\
p^{t,x_i}(T)= -g^{t,x_i}_y (T)-E[g^{t,x_i}_z(T)],
\end{array}
\right.
\end{equation}
where for $\zeta=y,z$,
\begin{eqnarray*}
H^{t,x_i}_{\zeta}(s)&=&b_{\zeta}(s,\hat X^{t,x_i}_i(s),E[\hat X^{t,x_i}_i(s)], \hat u_i)p^{t,x_i}(s)+
 \sigma_{\zeta}(s,\hat X^{t,x_i}_i(s),E\hat [X^{t,x_i}_i(s)])q^{t,x_i}(s) \\
&& - h_{\zeta}(s,\hat  X^{t,x_i}_i, E[\hat  X^{t,x_i}_i],\bar X(s),\hat  u_i),
\end{eqnarray*}
for which
\begin{equation}\label{ppH2}
\begin{array}{ll}
   H(t,x_i,v,p^{t,x_i}(t),q^{t,x_i}(t))-H(t,x_i,\hat u_i(t),p^{t,x_i}(t),q^{t,x_i}(t)) \le 0, \\ \qquad\qquad\qquad\qquad  \,\,\,   \forall v\in U,\; x_i\in \dbR, \;\, \ae t\in[0,T],\;\Prob-a.s.
\end{array}
\end{equation}
The closed-loop equilibrium state  associated to $\hat u_i$ of the $i$-th player is given by
\begin{equation}
d\hat X_i(s)= b(s,\hat X_i(s), E[\hat X_i(s)], \hat u_i(s))ds +\sigma(s, \hat X_i(s), E[\hat X_i(s)]) dW_i(s).
\end{equation}
 Since $\hat u_i$ does not depend on the Brownian motions of the other $N-1$ players, it is a decentralized strategy, i.e.,\@ the processes $\{\hat u_k, 1\leq k\leq N\}$ are independent.
Further, we  impose
\bcd\label{cond4}
 All the processes $\{\hat u_k, 1\leq k\leq N\}$ have the same law.
\ecd
This  restriction ensures that  $\{\hat X_i, 1\leq i\leq N\}$ are i.i.d. random processes. By the Law of Large Numbers, the consistency condition on $\bar X$  reads
\begin{equation}\label{consis}
\bar X(s)= E[\hat X_1(s)], \quad \forall s\in [0,T].
\end{equation}

A question of central interest is  how to characterize the performance of the set of strategies $\hat u=(\hat u_1, \ldots, \hat u_N)$ when they are implemented and assessed according to the original cost functionals $\{J^{i,N}, 1\leq i\leq N\}$.  An answer is provided in the following theorem whose proof is displayed in the next section. This is the second main result of the paper.

\begin{theorem}\label{theorem:asy} Under Assumptions (\ref{cond3}) and (\ref{cond4}), we have
\be\label{MF-app}
 J ^{i,N}(t,x_i;\hat u )-J^{i,N}(t,x_i;\hat u_{-i}, u_i^\varepsilon)
=\bar J^{i}(t,x_i; \hat u_i)-\bar J^{i}(t,x_i; u_i^\e )+O\l(\frac{\e}{\sqrt{N-1}}\r).
\ee
Moreover,
  $\hat u=(\hat u_1, \ldots, \hat u_N)\in \prod_{i=1}^N\cU[0,T]$  is an asymptotic sub-game perfect Nash equilibrium for the system (\ref{N-SDEu})-(\ref{N-J-1}) where $\delta_N=O(1/\sqrt{N})$.~\qed
\end{theorem}

\section{Proof of Theorem \ref{theorem:asy}}
\label{sec:proof}

This section is devoted to the proof of Theorem \ref{theorem:asy}. We first establish some performance estimates which will be used to  conclude the proof of the theorem.
\subsection{The performance estimate}
 We have
\begin{equation*}
\begin{array}{lll}
J^{i,N}(t,x_i; \hat u) = E_{t,x_i}\l[\int_t^T h(s, \hat X^{t,x_i}_i(s),  E [\hat X^{t,x_i}_i(s)] , \hat X^{(-i)}(s), \hat u_i(s))ds \right. \\ \qquad\qquad\qquad \qquad\qquad\qquad\qquad\qquad\qquad\left. +g(\hat X^{t,x_i}_i(T), E [\hat X^{t,x_i}_i(T)],  \hat X^{(-i)}(T))\r].
\end{array}
\end{equation*}

Now we fix $i\in \{1, \ldots, N\}$ and change $\hat u_i$ to $u_i^\varepsilon$ when all other players apply $\hat u_{-i}$, where
$$
u_i^{\varepsilon}(s):=\left\{
\begin{array}{ll}
u_{i}(s),  \,\, &  s\in [t, t+\varepsilon],\\
\hat u_i(s), \,\, & s\in [t,T]\backslash [t, t+\varepsilon],
\end{array}\right.
$$
and $u_i\in {\cal U}[0,T]$. We have
\be\begin{array}{lll}
  J^{i,N}(t,x_i;\hat u_{-i}, u_i^\e)= E\l[\int_t^T h\left(s,  X^{t,x_i}_i(s),  E [ X^{t,x_i}_i(s)] , \hat X^{(-i)}(s)
,  u^{\e}_i(s)\right)ds  \right. \\\qquad\qquad\qquad \qquad\qquad\qquad\qquad\qquad\qquad \left.
 +g\left( X^{t,x_i}_i(T), E [ X^{t,x_i}_i(T)],  \hat X^{(-i)}(T)  \right)\r],
\end{array}
\ee
where $X^{t,x_i}_i$ is the solution of \eqref{N-SDEu-flow} with admissible strategy $u_i^\varepsilon$. The following estimates will be frequently used in the sequel.

\begin{lemma} \label{lemma:m2bound}
For the $i$-th player, let $ X_i$ and $ \hat X_i$ be the state processes corresponding to  $u_i^\varepsilon$ and $\hat u_i$ respectively.  Then
 $$
E\Big[\sup_{t\le s\le T} |X^{t,x_i}_i(s)-\hat X^{t,x_i}_i(s)|^2\Big]\le C\varepsilon^{2},
 $$
 where $C$ does not depend on $(t, x_i)$.
\end{lemma}

\proof
 Using the  SDEs  (\ref{N-SDEu-flow}) for the two state processes, we have
\begin{eqnarray*}
X_i^{t,x_i}(\tau)-\hat X_i^{t,x_i}(\tau)&=& \int_t^\tau\left\{ b\left(s, X_i^{t,x_i}(s), E[ X_i^{t,x_i}(s)],  u_i^\varepsilon(s)\right) - b\left(s, \hat X_i^{t,x_i}(s), E[ \hat X_i^{t,x_i}(s)], \hat u_i(s)\right)\right\}ds \\
&&+\int_t^\tau \left\{\sigma\left(s,  X_i^{t,x_i}(s), E[ X_i^{t,x_i}(s)]\right)-\sigma\left(s, \hat X_i^{t,x_i}(s), E[\hat X_i^{t,x_i}(s)]\right)\right\} dW_i(s).
\end{eqnarray*}
By Burkholder-Davis-Gundy's inequality, we have
\begin{equation*}
\begin{array}{lll}
 E[\sup_{t\le \tau \le T}|X_i^{t,x_i}(\tau)-\hat X_i^{t,x_i}(\tau)|^2]\\ \quad
 \le C E  \left[\Big(\int_t^T\left| b\left(s, X_i^{t,x_i}(s), E[ X_i^{t,x_i}(s)],  u_i^\varepsilon(s)\right) - b\left(s, \hat X_i^{t,x_i}(s), E[ \hat X_i^{t,x_i}(s)], \hat u_i(s)\right)\right|ds\Big)^2\right]\\
\quad + CE\l[\int_t^T \left|\sigma\left(s,  X_i^{t,x_i}(s), E[ X_i^{t,x_i}(s)]\right)-\sigma
\left(s, \hat X_i^{t,x_i}(s), E[\hat X_i^{t,x_i}(s)]\right)\right|^2 ds\r] \\
\quad =:  C(I_b+I_\sigma),
\end{array}
\end{equation*}
where $C$ is a positive constant.

Noting that, in view of  Assumption (\ref{cond3}-(i)), if the positive constant $C$ denotes the bound of $b$, we have
$$
\begin{array}{lll}
|b\left(s, \hat X_i^{t,x_i}(s), E[\hat  X_i^{t,x_i}(s)],  u_i^\varepsilon(s)\right) - b\left(s, \hat X_i^{t,x_i}(s), E[ \hat X_i^{t,x_i}(s)], \hat u_i(s)\right)|\\ = |b\left(s, \hat X_i^{t,x_i}(s), E[\hat  X_i^{t,x_i}(s)],  u_i(s)\right) - b\left(s, \hat X_i^{t,x_i}(s), E[ \hat X_i^{t,x_i}(s)], \hat u_i(s)\right)|\ind_{[t,t+\e]}(s)\\
\le C\ind_{[t,t+\e]}(s),
\end{array}
$$
Thus, since $b$ is Lipschitz in $(y,z)$, by  Assumption (\ref{cond3}-(ii)), we have
\be\label{I-b-1}\begin{array}{lll}
| b\left(s, X_i^{t,x_i}(s), E[ X_i^{t,x_i}(s)],  u_i^\varepsilon(s)\right) - b\left(s, \hat X_i^{t,x_i}(s), E[ \hat X_i^{t,x_i}(s)], \hat u_i(s)\right)|\\ \le |b\left(s, X_i^{t,x_i}(s), E[ X_i^{t,x_i}(s)],  u_i^\varepsilon(s)\right) - b\left(s, \hat X_i^{t,x_i}(s), E[ \hat X_i^{t,x_i}(s)], u^\varepsilon_i(s)\right)|\\ +|b\left(s, \hat X_i^{t,x_i}(s), E[\hat  X_i^{t,x_i}(s)],  u_i^\varepsilon(s)\right) - b\left(s, \hat X_i^{t,x_i}(s), E[ \hat X_i^{t,x_i}(s)], \hat u_i(s)\right)|\\
 \le C\left(|X_i^{t,x_i}(s)-\hat X_i^{t,x_i}(s)|+E[|X_i^{t,x_i}(s)-\hat X_i^{t,x_i}(s)|]+\ind_{[t,t+\e]}(s)\right).
\end{array}
\ee
 The Cauchy-Schwarz inequality yields
\begin{eqnarray}
I_b&\le& C(T-t)\int_t^T  E[|X_i^{t,x_i}(s)-\hat X_i^{t,x_i}(s)|^2] ds
+C E\Big[ \Big(\int_t^T \ind_{[t,t+\e]}(s)ds\Big)^2 \Big] \nonumber \\
&\le & C\int_t^T  E[\sup_{t\le \eta\le s}|X_i^{t,x_i}(\eta)-\hat
 X_i^{t,x_i}(\eta)|^2]ds +C\e^2.
 \label{I-b-2}
\end{eqnarray}
In a similar fashion,  since $\sigma$ is Lipschitz in $(y,z)$, by  Assumption (\ref{cond3}-(ii)), we obtain
\begin{eqnarray}
I_\sigma &\le& C\int_t^T E\Big[\sup_{t\le \eta \le s}|X_i^{t,x_i}(\eta)-\hat X_i^{t,x_i}(\eta)|^2\Big]ds.
\end{eqnarray}
Therefore,
\begin{equation*}\begin{array}{lll}
E\big[\sup_{t\le \tau \le T}|X_i^{t,x_i}(\tau)-\hat X_i^{t,x_i}(\tau)|^2\big] \le C\int_t^T \Big[ E[\sup_{t\le \eta\le s }|X_i^{t,x_i}(\eta)-\hat  X_i^{t,x_i}(\eta)|^2]ds+C\e^2.
\end{array}
  \end{equation*}
The lemma follows from Gronwall's lemma.\qed

\begin{lemma} \label{lemma:xh}
We have
$$
E[\sup_{0\le s\le T} | \hat X_i(s)|^2 ]\le C E \Big[|\hat X_i(0)|^2+1\Big].
$$
\end{lemma}

\proof We write
\begin{eqnarray}
\hat X_i(s)=\hat X_i(0) + \int_0^s b(\tau, \hat X_i(\tau), E[\hat X_i(\tau)], \hat u_i(\tau) )d\tau + \int_0^s \sigma(\tau, \hat X_i(\tau), E[\hat X_i(\tau)] )dW_i(\tau).
\end{eqnarray}
Then, by Burkholder-Davis-Gundy's inequality, we have
\begin{eqnarray*}
E[\sup_{0\le s\le T}|\hat X_i(s)|^2 ]&\le&
C\big(E|\hat X_i(0)|^2+
 E\Big[\int_0^T|b(s, \hat X_i(s), E[\hat X_i(s)], \hat u_i(s) )|ds\Big]^2\big)\\
& &+ CE \int_0^T |\sigma(s, \hat X_i(s), E[\hat X_i(s)] )|^2ds.
\end{eqnarray*}
By the Lipschitz condition on $b$ and $\sigma$ (their derivatives w.r.t $(y,z)$ being bounded), we further obtain
\begin{equation*}\begin{array}{lll}
E[\sup_{0\le s\le T}|\hat X_i(s)|^2] \le C\big(E|\hat X_i(0)|^2 +1 +\int_0^TE[\sup_{0\le  \eta\le s}|\hat X_i(\eta)|^2\big)
\end{array}
\end{equation*}
 which combined with Gronwall's lemma yields the desired estimate. \qed

\begin{corollary} \label{lemma:xx}
We have, for $N\ge 2$,
 $$
\sup_{0\le s\le T} E[|\hat X^{(-i)}(s)- \bar X(s)|^2 ]\le \frac{C}{N-1},
$$
where $C$ does not depend on $N$.
\end{corollary}
\proof Thanks to Assumption (\ref{cond4}),  $\hat X_1, \ldots \hat X_N$ are i.i.d. processes. The estimate follows from Lemma \ref{lemma:xh}. \qed

\subsection{Proof of Theorem \ref{theorem:asy}}

In order to estimate $ J^{i,N}(t,x_i; \hat u)-J^{i,N}(t,x_i; \hat u_{-i}, u_i^\e   )$, we introduce some notation.
Let
\begin{eqnarray*}
\Delta_h(s)&=&   h\left(s, \hat X^{t,x}_i(s), E [\hat  X^{t,x_i}_i(s)] , \hat X^{(-i)}(s), \hat u_i(s)\right)   \\
 & & \qquad - h\left(s,  X^{t,x_i}_i(s),  E[ X^{t,x_i}_i(s)] , \hat X^{(-i)}(s) , u_i^\e (s)\right), \\
\Delta_g &=&g\left(\hat X^{t,x_i}_i(T), E[\hat X^{t,x_i}_i(T)],  \hat X^{(-i)}(T)\right)- g\left( X^{t,x_i}_i(T), E [X^{t,x_i}_i(T)],  \hat X^{(-i)}(T)  \right).
\end{eqnarray*}
We have
\begin{eqnarray*}
\Delta_h(s)  &= &\l[ h\l(s,\hat X^{t,x_i}_i(s), E[\hat X^{t,x_i}_i(s)],  \bar X(s),\hat u_i(s) \r) -h\l( s,X^{t,x_i}_i(s), E[ X^{t,x_i}_i(s)],  \bar X(s),u_i^\e(s)  \r)\r]\\
  &  &  +\Big\{ \l[ h\l(s,\hat X^{t,x_i}_i(s), E[\hat X^{t,x_i}_i(s)],  \hat X^{(-i)}(s), \hat u_i(s)  \r)-h\l( s,X^{t,x_i}_i(s), E[X^{t,x_i}_i(s)],  \hat X^{(-i)}(s),  u_i^\e(s)  \r)\r]  \\
  &  & \qquad -\l[h\l(s,\hat X^{t,x_i}_i(s), E[\hat X^{t,x_i}_i(s)],  \bar X(s), \hat u_i(s) \r)- h\l(s, X^{t,x_i}_i(s), E[X^{t,x_i}_i(s)],  \bar X(s), u^\e_i(s) \r)  \r] \Big\}\\
& =:& \Delta_{h1}  +\Delta_{h2}.
\end{eqnarray*}
Similarly,
\begin{eqnarray*}
\Delta_g  &= &\l[g\l(\hat  X^{t,x_i}_i(T), E[\hat X^{t,x_i}_i(T)],  \bar X(T)\r)-g\l( X^{t,x_i}_i(T), E[X^{t,x_i}_i(T)],  \bar X(T) \r)  \r]\\
  &  &  +\Big\{ \l[g\l(\hat X^{t,x_i}_i(T), E[\hat X^{t,x_i}_i(T)],  \hat X^{(-i)}(T)\r)-g\l( X^{t,x_i}_i(T), E [X^{t,x_i}_i(T)],  \hat X^{(-i)}(T)  \r) \r]  \\
  &  & \qquad -\l[ g\l(\hat X^{t,x_i}_i(T), E[\hat X^{t,x_i}_i(T)],  \bar X(T) \r)-g\l( X^{t,x_i}_i(T), E [X^{t,x_i}_i(T)],  \bar X(T)\r)  \r] \Big\}\\
& =:& \Delta_{g1}  +\Delta_{g2}.
\end{eqnarray*}
Now, noting that
\be\label{static}
 E\left[\int_t^T \Delta_{h1}(s) ds + \Delta_{g1}\right]= \bar J^{i}(t,x_i; \hat u_i)-\bar J^{i}(t,x_i; u_i^\e ),
\ee
the cost difference satisfies
\begin{equation}\label{perf-diff}
J ^{i,N}(t,x_i;\hat u )-J^{i,N}(t,x_i;\hat u_{-i}, u_i^\varepsilon)
=\bar J^{i}(t,x_i; \hat u_i)-\bar J^{i}(t,x_i; u_i^\e )+E \l[\int_t^T \Delta_{h2}(s)ds+\Delta_{g2}\r].
\end{equation}
We proceed to estimate
$$
E\l[\int_t^T \Delta_{h2}(s)ds +\Delta_{g2}\r].
$$

\begin{lemma}\label{estimate:A}
We have
\be
\l|E\l [\int_t^T \Delta_{h2}(s)ds+\Delta_{g2}\r]\r|\le \frac{C\e}{\sqrt{N-1}}.
\ee
\end{lemma}
\begin{proof}
 We will only estimate $E\l [\int_t^T \Delta_{h2}(s)ds\r]$. The second term may be handled in a similar fashion.
Let
$$
\alpha(w):=h(s,\hat X^{t,x_i}_i(s), E[\hat X^{t,x_i}_i(s)],  w, \hat u_i(s))- h( s, X^{t,x_i}_i(s), E[ X^{t,x_i}_i(s)],  w,  u^\e_i(s) )
$$
and
$$
\lambda(s):=\hat X^{(-i)}(s)-\bar X(s).
$$
Then we have
\be \label{deh}
\Delta_{h2}(s)=\alpha(\hat X^{(-i)}(s))-\alpha(\bar X(s))= \lambda(s) \int_0^1\alpha_w \l(\bar X(s)+ \theta \l[\hat X^{(-i)}(s)-\bar X(s) \r] \r)d\theta.
\ee
Noting that  by Assumption (\ref{cond3}-(iii)) on $h$, we may perform  similar calculations leading to (\ref{I-b-1})  to obtain
$$
|\alpha_w(w)|\le  C  \l[|X^{t,x_i}_i(s) -\hat X^{t,x_i}_i(s)|+ E[ | X^{t,x_i}_i(s)-  \hat X^{t,x_i}_i(s)| ]+ \ind_{[t,t+\e]}(s)\r].
$$
Therefore,
\begin{equation*}
|\Delta_{h2}(s)|\leq C |\lambda(s)| \l[|X^{t,x_i}_i(s) -\hat X^{t,x_i}_i(s)|+ E[ | X^{t,x_i}_i(s)-  \hat X^{t,x_i}_i(s)| ]+\ind_{[t,t+\e ]}(s)\r].
\end{equation*}
Therefore, by the Cauchy-Schwarz inequality,  we get
\begin{eqnarray*}
 E\int_t^T|\Delta_{h2}(s)|ds &\le& C\int_t^T (E[ |\lambda(s)|^2])^{1/2}\left( (E [|X^{t,x_i}_i(s) -\hat X^{t,x_i}_i(s)|^2])^{1/2}  + \ind_{[t,t+\e]}(s)\right) ds \\
&\le & C (\sup_{0\le s\le T}E |\lambda(s)|^2)^{1/2} (E [\sup_{t\le s\le T} |X^{t,x_i}_i(s) -\hat X^{t,x_i}_i(s)|^2])^{1/2}+\e).
\end{eqnarray*}
 Subsequently by Lemma \ref{lemma:m2bound} and Corollary  \ref{lemma:xx},
$$
\l| E\l[ \int_t^T \Delta_{h2}(s)ds\r]\r|
\le \frac{C \varepsilon}{ \sqrt{N-1}}.
 $$
\end{proof}

Finally, we combine  Lemma \ref{estimate:A} and the relation (\ref{perf-diff}) to conclude

$$
J^{i,N}(t,x_i; \hat u ) - J^{i,N}(t,x_{i};  \hat u_{-i}, u_i^\varepsilon  )=\bar J^i(t,x_i; \hat u_i ) - \bar J^i(t,x_{i};  u_i^\varepsilon  )+O\l(\frac{\e}{\sqrt{N-1}}\r).
$$
Furthermore, since $\hat u$ is determined by \eqref{ppH}-\eqref{consis},
$$
\lim_{\varepsilon\downarrow 0}\frac{\bar J^i(t,x_i; \hat u_i ) -\bar  J^i(t,x_{i}; u_i^\varepsilon  ) }{\varepsilon}\le 0,
$$
we finally get
$$
\lim_{\varepsilon\downarrow 0}\frac{J^{i,N}(t,x_i; \hat u ) - J^{i,N}(t,x_i;  \hat u_{-i}, u_i^\varepsilon  ) }{\varepsilon}\le\frac{ C}{\sqrt{N-1}}.
$$
This completes the proof of Theorem \ref{theorem:asy}. \qed

\section{A mean field LQG game}
\label{sec:mflq}

Consider a system of $N$ players. The dynamics  of the $i$-th player is given by \begin{equation}
dX_i(t)= (a X_i(s)+bu_i(s))ds +\sigma dW_i(s), \quad 1\leq i\leq N. \label{Xi}
\end{equation}
Denote $x=(x_1, \ldots, x_N)$ and $u=(u_1, \ldots, u_N)$. Its cost  functional at time
$t$ is
\begin{equation}
J^{i,N}(t, x_i;u)= \frac{1}{2} E \l[ \int_t^T u_i^2(s)ds \r]+\frac{\gamma}{2}
 E\left[X^{t,x_i}_i(T)-\Gamma_1 x_i-\Gamma_2 X^{(-i)}(T)\right]^2,\nonumber
\end{equation}
where $X^{(-i)}(t)=\frac{1}{N-1}\sum_{k\ne i}^N X_k(t)$. We take $\Gamma_1\neq 0$ and $\Gamma_2\neq 0$. A simple interpretation of the terminal cost is that each agent wants to adjust its terminal state based on its current state and also the mean field term $X^{(-i)}$ at time $T$.
 The cost  functional is time inconsistent.
Below, we will apply a consistent mean field approximation to construct a limiting  control problem.

 Following the scheme in Section \ref{sec:game},
 we introduce $\bar X_T$ as an approximation of $X^{(-i)}(T)$.
The new cost functional is
$$
\bar J^{i}(t, x_i;u_i)= \frac{1}{2} E\l[\int_t^T u_i^2(s)ds\r] +\frac{\gamma}{2} E
\left[X^{t,x_i}_i(T)-\Gamma_1 x_i-\Gamma_2 \bar X_T\right]^2.\nonumber
$$
This is a time-inconsistent control problem.
The same approach as in Section \ref{sec:lqr} can be applied.
The adjoint equation now reads
$$
\left\{
\begin{array}{l}
dp^{t,x_i}(s)= -a p^{t,x_i}(s) ds +q^{t,x_i}(s) dW_i(s),\\
p^{t,x_i}(T)= \gamma ( \Gamma_1 x_i+\Gamma_2 \bar X_T-X^{t,x_i}_i(T)).
\end{array}
\right.
$$

We look for a solution of the form
$$
p^{t,x_i}(s)= \b_s (\Gamma_1x_i+\Gamma_2 \bar X_T)-\alpha_s X^{t,x_i}_i(s) .
$$
The same set of ODEs is obtained as in Section \ref{sec:lqr}.
The equilibrium strategy  is given in the feedback form
\begin{equation}
\hat u_i(t)=b p^{t,x_i}(t)= -b (\a_t-\b_t\Gamma_1) x_i+b \b_t\Gamma_2 \bar X_T \label{ulqg}
\end{equation}
when the current state is $x_i$.
The closed loop equilibrium dynamics  of the $i$-th player is
\begin{equation}
d\hat X_i(s) =[a-b^2(\a_s-\b_s\Gamma_1)]\hat X_i(s)ds +b \b_s \Gamma_2\bar X_T ds +\sigma dW_i(s).  \label{xisg}
\end{equation}

Finally, we impose the consistency requirement. Assume all players have the same initial condition $y_0$, and so $\bar X_T$  can  be obtained as $E\hat X_i(T)$.
Now we take expectation in \eqref{xisg} to construct the ODE
$$
\dot m(s)=[a-b^2(\a_s-\b_s\Gamma_1)]m(s) +b \b_s \Gamma_2\bar X_T,\quad m(0)=y_0.  $$
By obvious notation for the transition function $\Phi$, we write the solution of the ODE as
$$
m(t)= \Phi(t,0)y_0+ \int_0^t \Phi(t,s) b \b_s \Gamma_2\bar X_T ds.
$$

Now the consistency condition for $\bar X$ becomes
$$
\bar X_T=  \Phi(T,0)y_0+ \int_0^T \Phi(T,s) b \b_s \Gamma_2\bar X_T ds.
$$
For this approach to have a solution for any given $y_0$, we need
\begin{equation}
 b  \Gamma_2\int_0^T \Phi(T,s) \b_sds \neq 1. \label{bga}
\end{equation}
If \eqref{bga} holds, we can solve $\bar X_T$ first and next determine the strategy \eqref{ulqg}.

\subsection{The performance difference}

Suppose \eqref{bga} holds.
For the performance estimate, we consider  the following set of admissible strategies
$$
 \mathcal{U}_0[0,T]:=\Big\{u: [0,T]\times \Omega\longrightarrow {\dbR}; \,\, u\, \mbox{ is}\,\,  \dbF\mbox{-adapted}, E[\esssup_{0\le s\le T}|u(s)|^2]<\infty \Big\},
$$
which is smaller than ${\cal U}[0,T]$.
The costs associated with $\hat u$ and $(u_i, \hat u_{-i})$ are, respectively, given by
\begin{eqnarray*}
J^{i,N}(t,x_i; \hat u) =\frac{1}{2} E\l[ \int_t^T \hat u_i^2(s)ds\r] +\frac{\gamma}{2} E
\left[\hat X^{t,x_i}_i(T)-\Gamma_1 x_i-\Gamma_2  \hat X^{(-i)}(T)\right]^2, \\
J^{i,N}(t,x_i;  u_i, \hat u_{-i}) = \frac{1}{2} E\l[ \int_t^T u_i^2(s)ds\r] +\frac{\gamma}{2} E
\left[X^{t,x_i}_i(T)-\Gamma_1 x_i-\Gamma_2 (\hat X^{(-i)}(T))\right]^2.
\end{eqnarray*}
The difference can be written as
\begin{eqnarray*}
J^{i,N}(t,x_i;  u_i, \hat u_{-i})-J^{i,N}(t,x_i; \hat u)& =& \frac{1}{2} E\l[ \int_t^T u_i^2(s)ds\r] +\frac{\gamma}{2} E
\left[X^{t,x_i}_i(T)-\Gamma_1 x_i-\Gamma_2 \bar X_T\right]^2\\
& &- \frac{1}{2} E\l[ \int_t^T \hat u_i^2(s)ds\r] -\frac{\gamma}{2} E
\left[\hat X^{t,x_i}_i(T)-\Gamma_1 x_i-\Gamma_2 \bar X_T\right]^2\\
 & & +d_{N},
\end{eqnarray*}
where
\begin{equation*}
d_N = \gamma \Gamma_2 E\[(\hat X^{t,x_i}_i(T)-X^{t,x_i}_i(T))(\hat X^{(-i)}(T)-\bar X_T)\].
  \end{equation*}
For any fixed $u_i\in {\cal U}_0[0,T],$ we can still prove Lemma \ref{lemma:m2bound}.  Corollary \ref{lemma:xx} also holds for $\hat u_j$, $1\le j\le N$.  We have
\begin{eqnarray*}
|d_N| &\le& \gamma\Gamma_2 (E|\hat X^{t,x_i}_i(T)-X^{t,x_i}_i(T)|^2)^{1/2} (E|\hat X^{(-i)}(T)-\bar X_T|^2)^{1/2} \\
&\le & \frac{C\varepsilon}{\sqrt{N-1}},
\end{eqnarray*}
where $C$ may depend on $u_i$. If $u_i\in {\cal U}[0,T]$ were considered, we would be unable to obtain the second inequality above.
 Finally,
\begin{equation*}
\lim_{\varepsilon\downarrow 0}\frac{J^{i,N}(t,x_i; \hat u)-J^{i,N}(t,x_i;  u_i, \hat u_{-i})}{\varepsilon}\le \frac{C}{\sqrt{N-1}}.
\end{equation*}
Thus, $\hat u$ is an asymptotic sub-game perfect Nash equilibrium point.
  \qed



\begin{thebibliography}{99}


\bibitem{AD} Andersson, D. and Djehiche, B. (2010): {\it A maximum principle for SDE's of mean-field type.} Appl. Math. Optim., {\bf 63}(3), 341-356.




\bibitem{B12}
Bardi, M. (2012): {\it Explicit solutions of some linear-quadratic mean field games},  Netw. Heterogeneous Media, {\bf 7}(2),  243-261.






\bibitem{BFY12}
 Bensoussan, A.,  Frehse, J., and  Yam P. (2012): {\it Overview on Mean Field Games and Mean Field Type Control Theory}, SpringerBriefs in Mathematics (to appear).



\bibitem{BSY} Bensoussan A., Sung,         K.C.J.,  and Yam, S.C.P. (2013):  {\it Linear-quadratic time-inconsistent mean field games}. Dynamic Games Appl., {\bf 3}(4), 537-552.


\bibitem{BSYY} Bensoussan, A., Sung, K.C.J., Yam, S.C.P., and Yung, S.P. (2011): {\it Linear-quadratic mean-field games}. Preprint.







\bibitem{BM} Bj\"ork, T. and Murgoci, A. (2008): {\it A general theory of Markovian time inconsistent stochastic control problems}. SSRN:1694759.

\bibitem{BMZ} Bj\"ork, T., Murgoci, A., and  Zhou, X.Y. (2014): {\it Mean-variance portfolio optimization with state-dependent risk aversion}.   Mathematical Finance, {\bf 24}(1), 1-24.

\bibitem{BCQ10}
Buckdahn, R., Cardaliaguet, P. and  Quincampoix, M. (2011):
{\it Some recent aspects of differential game theory}. Dynamic Games and Appl., {\bf 1}(1), 74-114.

\bibitem{BL} Buckdahn, R. and Li, J. (2008): {\it Stochastic differential games and viscosity solutions of Hamilton-Jocobi-Bellman-Isaacs equations}, SIAM J. Control Optim., {\bf 47}(1), 444-475.

\bibitem{BLP} Buckdahn, R., Li, J. and Peng, S. (2009): {\it Mean-field backward stochastic differential equations and related partial differential equations}. Stoch. Proc. and  Appl.,  {\bf 119}(10), 3133-3154.


\bibitem{BDL} Buckdahn, R., Djehiche, B., and Li, J. (2011): {\it A general stochastic maximum principle for SDEs of mean-field type.} Appl. Math. Optim. {\bf 64}(2), 197-216.



\bibitem{CD2012} Carmona, R. and Delarue, F. (2013): {\it Probabilistic analysis of mean-field games}. SIAM J. Control Optim., {\bf 51}(4), 2705-2734.

\bibitem{EL} Ekeland, I. and Lazrak, A. (2006): {\it Being serious about non-commitment: subgame perfect equilibrium in continuous time.} arXiv:math/0604264.

\bibitem{EP} Ekeland, I. and Pirvu, T.A. (2008): {\it Investment and consumption without commitment.}  Mathematics and Financial Economics, {\bf  2}, 57-86.


\bibitem{ELN13}
Elliott, R.J., Li, X., and Ni, Y.-H. (2013): {\it Discrete time mean-field stochastic linear-quadratic optimal control problems}. Automatica, {\bf 49}(11),  3222-3233.


\bibitem{Goldman} Goldman, S. M. (1980): {\it Consistent plans}. Rev. Financial Stud., {\bf 47}, 533-537.

\bibitem{GMS10}
Gomes, D.A., Mohr, J. and Souza, R.R. (2010):
{\it Discrete time, finite state space mean field games.}  J. Math.
Pures Appl.,  {\bf 93}, 308-328.

\bibitem{HJZ} Hu, Y., Jin, H., and Zhou, X.Y. (2012): {\it Time-inconsistent
stochastic  linear-quadratic control.} SIAM J. Control Optim. {\bf 50}(3), 1548-1572.


\bibitem{H10}
Huang, M. (2010): {\it Large-population LQG games involving a major player: the
Nash certainty equivalence principle.}  SIAM J. Control Optim.,
{\bf 48}(5), 3318-3353.

\bibitem{HCM03}
Huang, M., Caines, P.E., and Malham\'e, R.P. (2003): {\it Individual and mass
behaviour in large population stochastic wireless power control
problems: centralized and Nash equilibrium solutions.}
Proc. 42nd IEEE CDC, Maui, HI,   98-103.



\bibitem{HMC06}
Huang, M., Malham\'e, R.P., and Caines, P.E. (2006): {\it Large population stochastic dynamic games: Closed-loop McKean-Vlasov systems and the Nash certainty equivalence principle}. Communications in Information and Systems, {\bf 6}(3),  221-251.

\bibitem{HCM07}
Huang, M., Caines, P.E. and Malham\'e, R.P. (2007): {\it Large-population
cost-coupled LQG problems with nonuniform agents: individual-mass
behavior and decentralized $\varepsilon$-Nash equilibria.}  IEEE
Trans. Autom. Control,  {\bf 52}(9),  1560-1571.










\bibitem{KS}  Karatzas, I. and  Shreve, S.E. (1987): {\it Brownian Motion and Stochastic Calculus}, Springer-Verlag, New York.


\bibitem{KLY11}
Kolokoltsov, V.N.,  Li, J., and Yang, W. (2011): {\it  Mean field games and nonlinear Markov processes}.
Preprint.

\bibitem{LL07}
Lasry, J.-M. and Lions, P.-L. (2007): {\it Mean field games.} Japan. J.
Math.,  {\bf 2}(1), 229-260.

\bibitem{LZ08}
Li, T. and Zhang, J.-F. (2008): {\it Asymptotically optimal decentralized control
for large population stochastic multiagent systems.} IEEE Trans.
Automat. Control,  {\bf 53}(7),  1643-1660.


\bibitem{NC12}
 Nourian, M. and Caines, P.E. (2013):
 {\it $\epsilon$-Nash mean field game theory  for nonlinear stochastic dynamical systems
 with major and minor agents.}
  SIAM J. Control  Optim., {\bf 51}(4), 3302-3331.



\bibitem{Peleg-1} Peleg, B. and Menahem, E.Y. (1973): {\it  On the existence of a consistent course of action when tastes are changing}. Rev. Financial Stud., {\bf 40}, 391-401.



\bibitem{P} Peng, S. (1990): {\it A general stochastic maximum principle for optimal control problems.} SIAM J. Control and Optimization,  {\bf 28}(4), 966-979.

\bibitem{Peng07} Peng, S. (1997): {\it BSDE and stochastic optimizations}, in Topics in Stochastic Analysis, J. Yan, S. Peng, S. Fang, and L. Wu, eds., Science Press, Beijing, 1997 (in Chinese).


\bibitem{Phelps} Phelps, E.S. and Pollak, R.A. (1968): {\it On second-best national saving and
game-equilibrium growth.}  Review of Economic Studies, {\bf 35}, 185-99.

\bibitem{Poll} Pollak, R.A. (1968): {\it Consistent planning}. Rev. Financial Stud., {\bf 35}, 185-199.

\bibitem{Stoltz} Strotz R. (1955): {\it Myopia and inconsistency in dynamic utility maximization}. Rev. Financial Stud., {\bf 23}, 165-180.

\bibitem{TZB11}
Tembine, H., Zhu, Q. and Basar, T. (2011): {\it Risk-sensitive
mean-field stochastic differential games.}  Proc. 18th IFAC World Congress,
Milan, Italy.

\bibitem{WBR08}
Weintraub, G.Y., Benkard, C.L., and Van Roy, B. (2008): {\it Markov perfect
industry dynamics with many firms}.  Econometrica, {\bf 76}(6),
 1375-1411.

\bibitem{Y} Yong, J. (2013a): {\it  Linear-quadratic optimal control problems for mean-field stochastic differential equations.} SIAM J. Control Optim., {\bf 51}(4), 2809-2838.

\bibitem{Y13}
Yong, J. (2013b): {\it  Linear-quadratic optimal control problems for mean-field stochastic differential equations: Time-consistent solutions.} Preprint.

\bibitem{YZ} Yong, J. and  Zhou, X.Y.(1999): {\it Stochastic Controls: Hamiltonian Systems and HJB Equations.} Springer-Verlag, New York.

\bibitem{Zacc} Zaccour, G. (2008): {\it Time consistency in cooperative differential games: A tutorial.} INFOR,  {\bf 46}(1), 81-92.





















%

















\end{thebibliography}
\end{document}